\documentclass[11pt,letterpaper]{article}

\usepackage[utf8]{inputenc}
\usepackage{comment}
\usepackage{microtype}
\usepackage{graphicx}
\usepackage{appendix}
\usepackage{subcaption}
\usepackage{booktabs} % for professional tables
\usepackage{scalerel}
\usepackage[margin=1in]{geometry}
\usepackage[most]{tcolorbox}
\usepackage{amssymb}
\usepackage{amsthm}
\usepackage{amsmath}
\usepackage{mathtools}
\usepackage{nicefrac}
\usepackage{algorithm}
\usepackage{algpseudocode} % (algorithmicx)
\usepackage{tabularx}
\usepackage[table]{xcolor}

\newcommand{\cmark}{\ding{51}}
\usepackage{array}
\usepackage[hidelinks]{hyperref}

\usepackage[shortlabels]{enumitem}
\usepackage{dsfont}
\usepackage{url}
\usepackage{float}
\usepackage{pifont}
\usepackage{hhline}
\usepackage{multirow}
\usepackage{graphicx,wrapfig}
\usepackage{xcolor} % Include the xcolor package

\newcommand{\xmark}{\ding{55}}

\def\smskip{\smallskip}

\def\texitem#1{\par\smskip\noindent\hangindent 25pt
               \hbox to 25pt {\hss #1 ~}\ignorespaces}

\newcommand{\BEAS}{\begin{eqnarray*}}
\newcommand{\EEAS}{\end{eqnarray*}}
\newcommand{\BEA}{\begin{eqnarray}}
\newcommand{\EEA}{\end{eqnarray}}
\newcommand{\BEQ}{\begin{eqnarray}}
\newcommand{\EEQ}{\end{eqnarray}}
\newcommand{\BIT}{\begin{itemize}}
\newcommand{\EIT}{\end{itemize}}
\newcommand{\BNUM}{\begin{enumerate}}
\newcommand{\ENUM}{\end{enumerate}}

\newcommand{\BA}{\begin{array}}
\newcommand{\EA}{\end{array}}

\DeclareMathOperator*{\argmin}{\arg\!\min}

\newif\ifpagenumbering
\pagenumberingtrue

\pagenumberingfalse

\newtheorem{assumption}{Assumption}
\newtheorem{definition}{Definition}
\newtheorem{theorem}{Theorem}
\newtheorem{lemma}[theorem]{Lemma}
\newtheorem{proposition}[theorem]{Proposition}

\theoremstyle{remark}
\newtheorem{remark}{Remark}

\usepackage[numbers,square]{natbib}

\numberwithin{assumption}{section}
\numberwithin{definition}{section}
\numberwithin{theorem}{section}
\numberwithin{remark}{section}
\newcommand{\keywords}[1]{\textbf{Keywords:} #1}
\usepackage{hyperref}
\hypersetup{colorlinks,citecolor=blue,linktocpage,breaklinks=true}
\begin{document}
\title{ On the Analysis of Misspecified Variational
Inequalities with Nonlinear Constraints}

\author{
\parbox{\textwidth}{
\centering
Novel Kumar Dey\thanks{Department of Applied Mathematics, The University of Arizona, Tucson, AZ, USA.}
\quad
Mohammad Mahdi Ahmadi\thanks{Department of Systems and Industrial Engineering, The University of Arizona, Tucson, AZ, USA.\\ 
Emails: \texttt{\{ndey, ahmadi, erfany, afrooz\}@arizona.edu}}
\quad
Erfan Yazdandoost Hamedani$^{\thefootnote}$
\quad\\
Afrooz Jalilzadeh$^{\thefootnote}$
}
}

%\author{Novel Kumar Dey \thanks{Department of Applied Mathematics, The University of Arizona, Tucson, AZ, USA.\\} Mohammad Mahdi Ahmadi\thanks{Department of Systems and Industrial Engineering, The University of Arizona, Tucson, AZ, USA.\\ \qquad\{ndey, ahmadi, erfany, afrooz\}@arizona.edu}
%\quad Erfan Yazdandoost Hamedani$^*$ \quad Afrooz Jalilzadeh$^*$
%\vspace{5mm}
%}
\date{} % optional: removes date
\maketitle
\begin{abstract}
In this paper, we study a class of misspecified variational inequalities (VIs) where both the monotone operator and nonlinear convex constraints depend on an unknown parameter learned via a secondary VI. Existing data-driven VI methods typically follow a decoupled learn-then-optimize scheme, causing the approximation error from the learning to propagate the main decision-making problem and hinder convergence. We instead consider a simultaneous approach that jointly solves the main and secondary VIs. To efficiently handle nonlinear constraints with parameter misspecification, we propose a single-loop inexact Augmented Lagrangian method that simultaneously updates the primal decision variables, dual multipliers, and the misspecified parameter. The method combines a forward-reflected-backward step with an Augmented Lagrangian penalty, and explicitly handles misspecification on both the operator and constraint functions. Moreover, we introduce a relaxed performance metric based on the Minty VI gap combined with an aggregated infeasibility metric. By proving boundedness of the dual iterates, we establish $\mathcal{O}(1/K)$ ergodic convergence rates for these metrics. Numerical Experiments are provided to showcase the superior performance of our algorithm compared to
state-of-the-art methods. 
\end{abstract}


\section{Introduction}\label{sec:Intro}
Variational inequality (VI) problems have received significant attention in recent years owing to their general formulation that unifies a wide spectrum of well-known problems, including constrained and unconstrained optimization, saddle point (SP) problems, and Nash equilibrium problems \citep{facchinei2003finite}. These formulations naturally arise in diverse domains such as machine learning, economics, and game theory. Classical formulations assume complete knowledge of the problem parameters and feasible sets, enabling deterministic analysis of equilibria and their stability properties. However, in many practical applications, ranging from networked systems and energy markets to learning-based decision models, the underlying parameters governing the VI are not known precisely and must instead be inferred from data or auxiliary learning processes. This leads to \emph{misspecified variational inequalities}, where the true problem data depend on an unknown parameter vector that is estimated imperfectly. Such misspecification introduces an inherent coupling between the estimation and optimization components, giving rise to new analytical and computational challenges.

In contrast to stochastic or robust optimization approaches \citep{jiang2008stochastic,nemirovski2009robust,juditsky2011solving,koshal2012regularized} that assume access to distributional information or uncertainty sets, misspecified VIs capture settings in which the unknown parameters evolve through a separate, possibly \emph{data-driven}, mechanism. This perspective bridges classical equilibrium modeling with modern learning paradigms, allowing for adaptive decision-making under uncertainty without explicit probabilistic assumptions. While recent advances have established foundational results for misspecified VIs \citet{ahmadi2020resolution,jiang2016solution} in unconstrained or affine settings, the incorporation of nonlinear constraints remains largely unexplored.

In this paper, we address a class of \emph{misspecified constrained variational inequality} in which both the operator and the constraint functions depend on an unknown parameter vector $\theta^*$. Specifically, we seek $(x^*,\theta^*)$ such that
\begin{align}\label{eq:PVI}\tag{Misspecified VI}
F(x^*, \theta^*)^\top (y - x^*) \geq 0, \quad \forall y \in \mathcal{X}(\theta^*),
\end{align}
where $F : \mathbb{R}^n \times \mathbb{R}^m \to \mathbb{R}^n$ is a continuous mapping, and $\mathcal{X} \subseteq \mathbb{R}^n$ is the feasible set defined by
\begin{align}\label{eq2}
\mathcal{X}(\theta^*) \triangleq \left\{ x \in X \mid f_j(x,\theta^*) \leq 0, \quad \forall j = 1, \dotsc, J \right\},\quad \textit{ and }\quad X \triangleq \prod_{i=1}^{N} X_{i}.
\end{align}
Each $X_i \subseteq \mathbb{R}^{n_i}$ is a nonempty, closed, and convex set, and $f_j(\cdot,\theta^*)$  is convex in $x$. The parameter $\theta^* \in \Theta \subseteq \mathbb{R}^m$ is not known a priori; instead, it is determined as the solution of a secondary variational inequality
\begin{align}\label{eq:SVI}
H(\theta^*)^\top (\vartheta - \theta^*) \geq 0, \quad \forall \vartheta \in \Theta,
\end{align}
where $H : \mathbb{R}^m \to \mathbb{R}^m$ is a continuous and strongly monotone mapping, and $\Theta \subseteq \mathbb{R}^m$ is a nonempty, closed, and convex set. This coupled VI-within-VI framework simulates, for instance, network or market equilibria where the feasible set and operator are contingent upon demand or cost parameters derived from historical data, or policy-constrained Cournot competition with unobserved prices collected from empirical observations. Our focus is on algorithmic frameworks that \emph{simultaneously} update the primal variable $x$, the dual multipliers associated with the nonlinear constraints, and the parameter estimate $\theta$. To address the challenges posed by nonlinear constraints in the misspecified VI framework, we employ an augmented Lagrangian approach \citep{alizadeh2024randomized}. This method embeds the nonlinear constraints into the objective through a Lagrange multiplier term combined with a quadratic penalty, thereby enforcing feasibility progressively.

\subsection{Literature Review}
A wide body of work has examined VI problems, including foundational results for deterministic VIs \citep{tseng2000modified, nemirovski2004prox, solodov1996modified,solodov2000inexact} and their stochastic counterparts \citep{juditsky2011solving,nemirovski2009robust, lan2021mirror}. In this section, we review prevailing methodologies for optimization problems involving unknown or misspecified parameters, studied across frameworks such as constrained minimization, saddle point formulations, and VI models.

\textbf{Constraint-based Minimization.} In this line of research, numerous works focus on objective function minimization under both deterministic 
\citep{ahmadi2016rate,ahmadi2014data,aybat2021analysis} and stochastic frameworks 
\citep{jiang2016solution,yang2025analysis,yang2025data}. In the deterministic setting, \cite{ahmadi2014data} study a misspecified convex
optimization problem of the form $\min_{x \in X} f(x;\theta^*)$ subject to 
$\theta^* \in \arg\min_{\theta \in \Theta} g(\theta)$. When both the optimization
and learning objectives are strongly convex, the coupled update scheme achieves
a linear convergence rate. In contrast, when $f(\cdot;\theta^*)$ is merely convex
while $g$ remains strongly convex and a weak–sharpness condition is satisfied,
the standard $\mathcal{O}(1/K)$ rate is retained.
%but includes an additive biasterm proportional to the initial misspecification $\|\theta_0 - \theta^*\|$. 
\cite{aybat2021analysis} consider a convex optimization problem with a constraint set that depends on an unknown parameter and introduce an inexact
parametric augmented Lagrangian method (IPALM) to handle this misspecification.
The scheme employs an accelerated proximal-gradient step for updating the optimization decision variable 
and an outer update of the multipliers. They derive iteration-complexity bounds
showing that, under constant and increasing penalty sequences, IPALM requires
at most $\mathcal{O}(K^{-1/4})$ and $\mathcal{O}(\log K/K)$, respectively.

\textbf{Saddle Point (SP) and Variational Inequality (VI) Problems.} Recent research has extended the study of misspecification beyond classical optimization to SP and VI problems.  In \citep{ahmadi2020resolution}, a class of misspecified monotone VI problems is studied, where misspecifications may arise in the objective. To tackle these issues, the authors develop extragradient and regularized first-order algorithms tailored for misspecified monotone VI problems coupled with strongly convex learning subproblems, thereby guaranteeing stable convergence. Complementary to this, \citep{jiang2016solution} analyze stochastic variational inequality (SVI) problems in the context of misspecification, covering both stochastic convex optimization and a range of stochastic equilibrium models. Their analysis 
critically relies on the strong convexity of the learning problem, which guarantees convergence of both the optimization and 
learning iterates. They further establish almost-sure convergence in strongly monotone settings, achieving the optimal rate of 
$\mathcal{O}(1/K)$, and in merely monotone settings by employing iterative Tikhonov regularization, which yields a rate of 
$\mathcal{O}(\sqrt{\log K}/\sqrt{K})$ under weak sharpness. \cite{ahmadi2020resolution} extend their earlier work \cite{ahmadi2014data} to
misspecified monotone variational inequality problems and show that a
constant–steplength misspecified extragradient scheme achieves asymptotic
convergence of the iterates. More recently, \citet{ahmadi2025simultaneous} examine a class of misspecified convex-concave SP problems where the misspecified parameter can be obtained through a secondary SP problem with a strongly convex-concave objective function. They introduce a Learning-aware Accelerated Primal-Dual method that integrates the parameter updates directly into the dual momentum step and utilizes an adaptive backtracking scheme to select step sizes adaptively. Their approach attains a convergence rate of $\mathcal{O}(\log K / K)$ without relying on compactness of the feasible sets. In addition, they extend their method to settings where the secondary SP problem may not have a unique solution, proving that a modified version of the algorithm tailored to this structure achieves a rate of $\mathcal{O}(1/\sqrt{K})$ in terms of both gap function and learning suboptimality.

\subsection{Research Gap and Our Contribution}
Previous work on misspecified VIs has primarily considered settings where the constraint set is easy to project onto and does not depend on any unknown parameters. In contrast, this paper focuses on a broader class of misspecified VIs that involve misspecified nonlinear constraints in the primary problem. A detailed comparison between our method and existing approaches is provided in Table~\ref{tab:misspec_table}. 

More specifically, we study a coupled VI framework in which the primary VI \eqref{eq:PVI} is defined over a parameter-dependent feasible region $\mathcal{X}(\theta)$, specified by a set of convex nonlinear functions. The unknown parameter $\theta^*$ is simultaneously estimated by solving a secondary VI problem \eqref{eq:SVI}, which we assume to be strongly monotone. To address this problem, we propose a novel single-loop method that combines augmented Lagrangian and forward-reflected-backward step. To contend with the challenge of misspecified constraints, we show that under suitable assumptions, the corresponding dual iterates remain bounded, which helps us to establish the convergence rate of the proposed method. Furthermore, to measure the quality of the solution, we introduce a \emph{relaxed Minty VI gap} as well as \emph{infeasibility} metrics and establish an ergodic convergence rate of $\mathcal O(1/K)$ in a monotone setting, matching the optimal first-order convergence rate for fully specified monotone VIs \citep{juditsky2011solving,nemirovski2009robust}.

\begin{table}[H]
\centering
\caption{Comparison of related misspecified convex optimization/monotone VI studies.}
\label{tab:misspec_table}
\footnotesize
\setlength{\tabcolsep}{8pt}     % smaller spacing since columns expand automatically
\renewcommand{\arraystretch}{1}

\begin{tabularx}{\textwidth}{
 | >{\centering\arraybackslash}c
 | >{\centering\arraybackslash}X
 | >{\centering\arraybackslash}X
 | >{\centering\arraybackslash}X
 | >{\centering\arraybackslash}X
 | >{\centering\arraybackslash}X |
}
\hline
\rowcolor{gray!15}
References & Primary problem &  Non-linear const. & Misspecif. in const. & Secondary problem &Convergence rate \\
\hline
\citet{ahmadi2020resolution} & VI &  \xmark & \xmark & VI &- \\
\hline
\citet{aybat2021analysis} & Optimization & \cmark & \cmark & Optimization & $\mathcal{O}(\frac{\log K}{K})$ \\
\hline
\citet{jiang2016solution} & Stochastic VI &\xmark & \xmark & Stochastic VI & $\mathcal{O}\left(\frac{\sqrt{\log K}}{\sqrt{K}}\right)$ \\
\hline
\citet{ahmadi2025simultaneous} & Saddle Point &  \xmark & \xmark & Saddle Point &$\mathcal{O}(\frac{\log K}{K})$ \\
\hline
\textbf{This work} & \textbf{VI}  & \textbf{\cmark} & \textbf{\cmark} & \textbf{VI}&$\mathbf{\mathcal{O}(\frac{1}{K})}$\\
\hline
\end{tabularx}
\end{table}

\subsection{Motivating Example}\label{sec:motivating-example}
The Misspecified constrained VIs in problem \eqref{eq:PVI} have broad applications in markets and networked systems where feasibility depends on misspecified parameters learned from data. Some of these examples include constrained Nash-Cournot competition characterized by demand misspecification and policy cap \cite{jiang2011learning}, inverse and data-driven equilibrium in \cite{thai2018imputing}, misspecified Markowitz portfolio optimization problem in \cite{aybat2021analysis}, and minimax schemes with expectation constraints \cite{yang_Xu2025data}. Next, we provide a Cournot example in the following motivating example.

\paragraph{Cournot with a price-function constraint and unknown slope:} Modern commodity markets often operate under policy rules that constrain prices while firms face uncertainty about demand responsiveness. Consider a homogeneous goods market with $N$ firms, $i=1,\dots,N$. Firm $i$ chooses a quantity $x_i\ge0$; let the aggregate output be $X\triangleq\sum_{i=1}^N x_i$. The market clears at a linear inverse demand $p(X; b^*) = a^* - b^* X,$
where $a^*>0$ is known and $b^*>0$ is the unknown slope capturing how aggressively price falls with output. Firms know costs and capacities but do not know the true $b^*$. Beyond firm-level capacity limits $x_i\in[0,\mathrm{Cap}_i]$, many markets also face a price-function constraint induced by policy, such as an administrative price cap. We encode this as $p(X;b^*) \le a^*-\delta$, for some $\delta>0$. Mathematically, it imposes a lower bound on the feasible output that depends on the (unknown) slope $b^*$. Let $c_i:[0,\mathrm{Cap}_i]\to\mathbb{R}_+$ be firm $i$'s convex cost. Then, each firm's objective is to maximize the profit $\pi_i({\bf x}) \triangleq p(X; b^*)x_i - c_i(x_i)$ such that $x_i\in [0,\mathrm{Cap}_i]$ and $p(X;b^*) \le a^*-\delta$. In practice, the firms observe sample pairs $\{(p_{\text{obs},t}, X_t)\}_{t=1}^T$ and must learn $b^*$ from data while computing equilibria that respect the price constraint.

\paragraph{Formulation as a Misspecified Variational Inequality:} We cast the problem as finding a pair $(x^*,b^*)$ in the form of \eqref{eq:PVI} where the secondary VI \eqref{eq:SVI} represents learning parameter $b^*$ from $\{(p_{\text{obs},t},X_t)\}_{t=1}^T$. For instance, one can use least-squares regression to learn the unknown parameter $b^*$ as follows: 
\begin{align}\label{eq:LS}
    b^*= \argmin_{b\geq 0}\
\mathcal {L}(b)= \frac12\sum_{t=1}^T\big(p_{\text{obs},t}-(a^*-bX_t)\big)^2.
\end{align}
Moreover, the feasible set in \eqref{eq:PVI} can be characterized as $\displaystyle \mathcal X(b^*) \triangleq \Bigl\{ \mathbf{x}=[x_i]_{i=1}^N\in\prod_{i=1}^{N}[0,\text{Cap}_i] \Big| p(X;b^*) \le a^*-\delta \Bigr\}$. Then, the operator $F(\mathbf{x}, b)$ is defined by concatenating the gradients of the negative profit functions of the firms $F_i(\mathbf{x},b)\triangleq c_i'(x_i)+ b(X+x_i)- a^*$.

\paragraph{Outline of the paper:}
The remainder of the paper is organized as follows. In Section \ref{sec:prelim}, we introduce notations and state the assumptions. Section \ref{sec:algorithm} presents the proposed algorithm along with augmented-Lagrangian formulation of the misspecified variational inequality. In Section \ref{sec:theory}, we establish convergence theory for the proposed method and derive the explicit convergence rate. Section \ref{sec:numerics} implements a misspecified Cournot equilibrium and empirically compares the proposed scheme against two benchmark methods. Detailed proofs of the lemmas and the theorem are provided in the Appendix.

\section{Preliminaries and Assumptions}\label{sec:prelim}
\paragraph{Notation.} Throughout the paper, $\|\cdot\|$ denotes the Euclidean norm, and for any scalar $x \in \mathbb{R}$ we use $[x]_{+} \triangleq \max\{x,0\}$. For a differentiable vector-valued function $f : \mathbb{R}^n \to \mathbb{R}^m$, we write $\mathbf{J}f : \mathbb{R}^n \to \mathbb{R}^{m \times n}$ for its Jacobian matrix. For a closed convex set $X$ the projection operator is denoted by $\Pi_{X}(x)\triangleq \argmin_{y\in X}\|x-y\|$.

We make the following standard assumptions regarding \ref{eq:PVI} problem.
\begin{assumption}\label{assum:assum1}
The following hold:
\begin{enumerate}
    \item[(i)] For any $\theta \in \Theta$, the mapping $F(\cdot, \theta) : \mathbb{R}^n \to \mathbb{R}^n$ is continuous and monotone on $X$, i.e.,
    \[
    \left( F(x, \theta) - F(y, \theta) \right)^\top (x - y) \geq 0, \quad \forall x,y\in X.
    \]
    \item[(ii)] The mapping $H : \mathbb{R}^m \to \mathbb{R}^m$ is Lipschitz continuous and strongly monotone on $\Theta$, i.e.,
    \[
    \|H(\theta) - H(\vartheta)\|\leq L_H\|\theta-\vartheta\|,\quad 
    \left( H(\theta) - H(\vartheta) \right)^\top (\theta - \vartheta) \geq \mu_H\|\theta-\vartheta\|,\quad \forall \theta,\vartheta\in\Theta.
    \]
    \item [(iii)] For any $j\in\{1,\hdots,J\}$ and $\theta \in \Theta$, the function $f_j(\cdot,\theta)$ is convex and continuously differentiable on $X$. Moreover, there exists $L_{\nabla f,x}\ge 0$ such that for all $x,y\in X$ and all $\theta\in\Theta$,
    \begin{align*}
        \|\nabla_x f_j(x,\theta)-\nabla_x f_j(y,\theta)\| \leq L_{\nabla f,x}\|x-y\|.
    \end{align*}
    In addition, for any $\theta \in \Theta$ and $x\in X$, $f_j(\cdot, \theta)$ and $f_j(x,\cdot)$ are Lipschitz continuous with constants $L_{f,x}\ge 0$ and $L_{f,\theta}\ge 0$ respectively.
    \item[(iv)] The mapping $F: \mathbb{R}^n\times \mathbb{R}^m \to \mathbb{R}^n$ is Lipschitz continuous in $x$ and $\theta$ over $X \times \Theta$; that is, there exist constants $L_{F,x} \geq 0$ and $L_{F,\theta} \geq  0$ such that:
    \begin{align*}
    & \| F(x_1, \theta) - F(x_2, \theta) \| \leq L_{F,x} \| x_1 - x_2 \|, \quad \forall x_1, x_2 \in X, \quad \forall \theta \in \Theta,\\
    & \| F(x, \theta_1) - F(x, \theta_2) \| \leq L_{F,\theta} \| \theta_1 - \theta_2 \|, \quad \forall x \in X, \quad \forall \theta_1, \theta_2 \in \Theta.
    \end{align*} 
    \item[(v)] The sets $X \subseteq \mathbb{R}^n$ and $\Theta \subseteq \mathbb{R}^m$ are nonempty, closed, convex and compact.
    \item [(vi)] (Slater's Condition) There exists $\hat{x} \in X$ such that: $\displaystyle
    f_j(\hat{x}, \theta^*) < 0$, for all $j \in\{ 1, \dotsc, J\}$. 
\end{enumerate}
\end{assumption}

\begin{definition}\label{def:def_aug_lag}(Augmented-Lagrangian Function for Misspecified VI) For $x, y \in \mathbb{R}^n$, $\lambda \in \mathbb{R}^J$, $\theta \in \mathbb{R}^m$, and $\rho > 0$, the Augmented-Lagrangian function for the \ref{eq:PVI} problem is defined as
\[
\mathcal{L}_\rho(x, y, \theta, \lambda) \triangleq F(y, \theta)^\top (x - y) + \Phi_\rho(x, \lambda, \theta),
\]
where $\Phi_\rho(x,\lambda,\theta) \triangleq  \sum_{j=1}^J \phi_\rho(f_j(x, \theta), \lambda_j)$ and $\phi_\rho(u, v) \triangleq
\begin{cases}
uv + \frac{\rho}{2} u^2, & \text{if } \rho u + v \geq 0, \\
-\frac{v^2}{2\rho}, & \text{otherwise.}
\end{cases}$
\end{definition}
Based on the above definition, one can observe that 
\begin{equation}\label{rem:rem6_prop2}
    \nabla_x \Phi_{\rho_k}(x_k, \lambda_k, \theta^*) = \sum_{j=1}^J \left[ \rho_k f_j(x_k, \theta^*) + \lambda_k^{(j)} \right]_+ \nabla f_j(x_k, \theta^*).
\end{equation}

\begin{proposition}\label{prop1}(KKT Conditions) Suppose that a solution pair $(x^*,\theta^*)$ of (\ref{eq:PVI}) exists and Assumption \ref{assum:assum1} holds. Let $f(x, \theta^*) \triangleq [f_1(x, \theta^*), \dots, f_J(x, \theta^*)]^\top$, and the Jacobian matrix $\nabla f(x, \theta^*) \triangleq [\nabla f_1(x, \theta^*), \dots, \nabla f_J(x, \theta^*)]^\top \in \mathbb{R}^{n \times J }$. There exists a pair $(x^*, \theta^*) \in \mathbb{R}^n \times \mathbb{R}^m$ and $\lambda^* \in \mathbb{R}^J$ satisfying the following KKT conditions:
\begin{itemize}
    \item[(i)] Stationarity:  $0 \in F(x^*, \theta^*) +  \nabla f(x^*, \theta^*)^\top \lambda^* + \mathcal{N}_X(x^*) \quad \textit{ and } \quad 0 \in H(\theta^*) + \mathcal{N}_\Theta(\theta^*),$
where $\mathcal{N}_X(x^*)$, and $\mathcal{N}_\Theta(\theta^*)$ are the normal cones of sets $X$ and $\Theta$ at $x^*$ and $\theta^*$, respectively.
\item[(ii)] Complementary Slackness:  $0 \leq \lambda^* \perp -f(x^*, \theta^*) \geq 0.$
\item[(iii)] Feasibility:  $x^* \in \mathcal{X}(\theta^*) \triangleq \{x \in X \mid f_j(x, \theta^*) \leq 0, \quad \forall j = 1, \dotsc, J\}, \quad \textit{and }\quad \theta^* \in \Theta.$
\end{itemize}
\end{proposition}
\proof{Proof.} 
Note that any solution pair $(x^*,\theta^*)$ of \eqref{eq:PVI} is a solution of the following optimization problem
\begin{equation}\label{eq:KKT}
\min _{y \in X} y^{T} F(x^*,\theta^*) \quad\text { s.t. } \quad f(y,\theta^*) \leq \mathbf{0}.
\end{equation}
Since the Slater condition holds, the first-order KKT condition of (\ref{eq:KKT}) implies that there exists $\lambda^{*} \in \mathbb{R}^{J}$ satisfying conditions (i)-(iii). Moreover, the optimality condition of the secondary VI problem \eqref{eq:SVI} immediately implies that $0\in H(\theta^*)+\mathcal N_\Theta(\theta^*)$.
\endproof

\begin{lemma}\label{lemma:lem1}
    Consider \eqref{eq:PVI} problem under Assumption \ref{assum:assum1} and let $(x^*,\theta^*,\lambda^*)$ satisfy the KKT
conditions of Proposition \ref{prop1}. Then for all $x\in \mathcal{X(\theta^*)}$, $(x-x^*)^\top F(x^*,\theta^*) +  \sum_{j=1}^{J} \lambda^{*(j)} f_j(x,\theta^*)\ge 0$. 
    % \begin{align*}
    %     (x-x^*)^\top F(x^*,\theta^*) +  \sum_{j \in J} \lambda^{*(j)} f_j(x,\theta^*\ge 0, \qquad \text{ for all } \quad x\in \mathcal{X}(\theta^*).
    % \end{align*}
    \begin{proof}{Proof}
        The result follows directly from Lemma 1 in \cite{alizadeh2024randomized}, which establishes the corresponding inequality under Assumption \ref{assum:assum1}. 
    \end{proof}
\end{lemma}

\section{Proposed Method}\label{sec:algorithm}
In this section, we present an augmented Lagrangian method for solving \eqref{eq:PVI}. Due to the presence of a nonlinear misspecified constraint, we introduce a multiplier $\lambda$ to relax this constraint, following Definition \ref{def:def_aug_lag}. Moreover, to handle the parameter misspecification, we simultaneously solve a secondary VI problem that generates a sequence of iterates converging to the unique optimal parameter $\theta^*$. The approximation of $\theta^*$ obtained at each iteration is then used to update both the main decision variable $x$ and the multiplier $\lambda$. To evaluate the quality of the obtained solution, we introduce a \emph{relaxed gap function}. First, recall that the standard VI gap function for \eqref{eq:PVI} is defined as $\text{Gap}(x, \theta^*) \triangleq \sup_{y \in \mathcal{X}(\theta^*)} { F(y, \theta^*)^\top (x - y) }, $ and an $\epsilon$-approximate solution satisfies $\text{Gap}(x, \theta^*) \leq \epsilon$. This gap function is a valid optimality metric, since $\text{Gap}(x, \theta^*) \geq 0$ for any feasible $x \in \mathcal{X}(\theta^*)$, and $\text{Gap}(x, \theta^*) = 0$ implies that $x$ is a solution to \eqref{eq:PVI}. However, because our proposed method employs a dual multiplier to enforce constraints and accounts for the misspecified parameter $\theta$, the generated sequence may be infeasible. Consequently, the algorithm’s output $\bar{x}$ may not belong to the feasible set $\mathcal{X}(\theta^*)$, and we might have $\text{Gap}(\bar{x}, \theta^*) < 0$. To address this issue, we measure both (i) the level of \emph{infeasibility}, quantified as $\mathbf{1}^\top [f(x, \theta^*)]_{+} \leq \epsilon$, and (ii) a \emph{relaxed gap function} defined as  $\widetilde{\text{Gap}}(x, \theta^*) \triangleq \sup_{y \in \mathcal{X}_\epsilon(\theta^*)} { F(y, \theta^*)^\top (x - y) }$ where $\mathcal{X}_\epsilon(\theta^*) \triangleq \{ x \in X \mid \mathbf{1}^\top [f(x, \theta^*)]_{+} \leq \epsilon \}$ denotes the $\epsilon$-enlargement of the feasible set. It is straightforward to verify that $\widetilde{\text{Gap}}(x, \theta^*) \geq 0$ for any $x \in \mathcal{X}_\epsilon(\theta^*)$, and $\widetilde{\text{Gap}}(x, \theta^*) \geq \text{Gap}(x, \theta^*)$ for any $x$. Therefore, we define an $\epsilon$-approximate solution $\bar{x} \in X$ if 
\begin{equation*}
    \text{(Relaxed gap): } \widetilde{\text{Gap}}(\bar x, \theta^*)\leq \epsilon,\qquad \text{(Infasibility): } \mathbf{1}^\top [f(\bar x,\theta^*)]_{+}\leq \epsilon.
\end{equation*} 
Furthermore, as $\epsilon \to 0$, the point $\bar{x}$ becomes feasible, i.e., $\bar{x} \in \mathcal{X}(\theta^*)$, which implies that $\widetilde{\text{Gap}}(x, \theta^*) = \text{Gap}(x, \theta^*) = 0$, hence, $\bar{x}$ is a solution to \eqref{eq:PVI}.

To find such an approximate solution, we propose a novel single-loop algorithm by combining the forward-reflected-backward method with an augmented Lagrangian update. This combination enables us to effectively handle nonlinear constraints under misspecified evaluations of both the operator $F$ and the nonlinear constraint functions $f_j$’s. 
Recall the augmented Lagrangian function from Definition \ref{def:def_aug_lag}, then for fixed $(\theta,\lambda)$, the term $\Phi_\rho(x,\lambda,\theta)$ imposes penalization for violating the constraints $f_j(x,\theta)\leq 0$. In particular, at each iteration $k\geq 0$, the decision variable $x$ is updated using a projected forward operation applied to the augmented Lagrangian function combined with a reflected (correction) step $r_k= F(x_k,\theta_k)-F(x_{k-1},\theta_{k-1})$ as $x_{k+1}=\Pi_X[x_k-\gamma_k (r_k+\nabla_x \mathcal L_{\rho_k}(x_k,\lambda_k,\theta_k))]$. Then, the dual multiplier $\lambda_k$ is updated via a projected gradient ascent step to promote feasibility based on the newly updated primal variable $x_{k+1}$. Finally, the misspecified parameter $\theta_{k+1}$ is refined by performing one projected gradient step associated with the secondary VI problem \eqref{eq:SVI}. The steps of the algorithm are outlined in Algorithm \ref{alg:alm_mvi}. 

%We now present the proposed algorithm below:
\begin{algorithm}[H]
\caption{Augmented Lagrangian Method for Misspecified VI (ALM-Misspecified VI)}\label{alg:alm_mvi}
\begin{algorithmic}[1]
\State \textbf{Input:} $x_0 \in X$, $\theta_0 \in \Theta$, $\{\rho_k,\gamma_k,\eta_k\}_{k\geq 0}\subset \mathbb R_{++}$.
\State \textbf{Initialization:} $x_{-1}\gets x_0$, $\theta_{-1}\gets \theta_0$, $\lambda_0 \gets \mathbf{0}$ 
\For{$k = 0, 1, 2, \dots, K-1$}
    \State $r_k \gets F(x_k,\theta_k)-F(x_{k-1},\theta_{k-1})$
    \State $x_{k+1}
\gets \Pi_{X}\left[ x_k
-\gamma_k \left(
F(x_k,\theta_k)+ r_k + \mathbf{J}f(x_k,\theta_k) [\rho_k f(x_k,\theta_k)+\lambda_k]_+
\right)
\right]$
    \State $\lambda_{k+1} \gets \left[ \lambda_k + \rho_k f(x_{k+1}, \theta_k) \right]_+$
    \State $\theta_{k+1} \gets \Pi_\Theta \left[ \theta_k - \eta_k H(\theta_k) \right]$
\EndFor
\end{algorithmic}
\end{algorithm}

\section{Convergence Analysis}\label{sec:theory}
In this section, we establish the convergence properties of Algorithm \ref{alg:alm_mvi} under Assumption \ref{assum:assum1}. In particular, through a step-by-step analysis, we demonstrate that our method achieves a convergence rate of $\mathcal{O}(1/K)$ in terms of the relaxed gap and infeasibility metrics. All related proofs are provided in the Appendix. 

We begin by showing some preliminary bounds on the difference between the \emph{constraint enforcement term} of the augmented Lagrangian function, $\Phi_{\rho_k}(x_{k+1},\lambda_k,\theta^*)$, and the corresponding term in the standard Lagrangian, $\sum^J_{j=1}\lambda^{(j)}f_j(x_{k+1},\theta^*)$.

\begin{lemma}\label{lemma:lemma2}
Consider Algorithm \ref{alg:alm_mvi}. Let us define $J^+_{k} \triangleq \Big\{ j\in[J] : \rho_k f_j(x_{k+1},\theta_k) + \lambda_k^{(j)} \ge 0 \Big\} \text{ and }
J^-_{k} \triangleq [J]\setminus J^+_{k}$. Then for any $\lambda\in\mathbb{R}^J_+$, the following inequality holds:
\begin{align*}
&-\Phi_{\rho_k}\bigl(x_{k+1},\lambda_k,\theta^*\bigr)
+ \sum_{j=1}^J\lambda^{(j)}f_j(x_{k+1},\theta^*) \leq
\frac{1}{2 \rho_k}\left( \|\lambda_k - \lambda\|^2 - \|\lambda_{k+1}-\lambda\|^2\right) +\frac{J \rho_k}{2} L^2_{f,\theta} \|\theta_k - \theta^*\|^2  \notag\\
& \qquad  + L^\Phi_{\lambda\theta}\|\lambda_{k+1}-\lambda\| \|\theta_k-\theta^*\| + 2L_{f,\theta}(J\rho_k D_f+\sqrt{J}\|\lambda_k\|)\|\theta_k-\theta^*\|.
\end{align*}
\end{lemma}
\begin{proof}{Proof}
See Appendix \ref{sec:lem2} for the proof. 
\end{proof}

Based on the result of Lemma \ref{lemma:lemma2}, we next establish a one-step bound for the iterates generated by the proposed algorithm.

\begin{lemma}[Lipschitz continuity of $\nabla_x \Phi_{\rho}(\cdot, \lambda, \theta)$]\label{lemma:lip_grad} Suppose Assumption \ref{assum:assum1} holds. Then, for every $\rho>0$, every $\lambda\in\mathbb R^J$, and every $\theta\in\Theta$, the mapping $\nabla_x\Phi_\rho(\cdot,\lambda,\theta)$ is Lipschitz continuous on $X$. In particular, for any $x,y\in X$,
%\begin{align*}
$    \|\nabla_x \Phi_{\rho}(x,\lambda,\theta)-\nabla_x \Phi_{\rho}(y,\lambda,\theta)\|
\leq
(\rho C_1+\|\lambda\|C_2)\|x-y\|,$ 
%\end{align*}
where $C_1\triangleq \sqrt{J}\big(L_{\nabla f,x}D_f+L_{f,x}M_{\nabla f}\big),
\quad
C_2\triangleq \sqrt{J}\,L_{\nabla f,x},$ and $D_f\triangleq \sup_{(z,\vartheta)\in X\times\Theta}\|[f(z,\vartheta)]_+\|<\infty,
\quad
M_{\nabla f}\triangleq \sup_{(z,\vartheta)\in X\times\Theta}\|\nabla_x f(z,\vartheta)\|<\infty.$
\end{lemma}

\begin{proof}{Proof}
See Appendix \ref{sec:lem3} for the proof. 
\end{proof}

\begin{lemma}[One-step analysis]
    \label{prop:prop_1}
    Let $\{(x_k, \lambda_k, \theta_k)\}_{k \geq 0}$ be the sequence generated by Algorithm \ref{alg:alm_mvi} and suppose Assumption \eqref{assum:assum1} holds. Then, for any $x \in X$, $\lambda \in \mathbb{R}^J_+$, the following inequality holds for any $k\geq 0$, 
\begin{align*}
     &\left(x_{k+1}-x\right)^{\top} F\left(x,\theta^*\right) + \sum_{j=1}^J\lambda^{(j)}f_j(x_{k+1},\theta^*) - \Phi_{\rho_k}(x,\lambda_k,\theta^*)  \notag\\
    & \quad \leq \frac{1}{2\gamma_k}\Big(\|x_k - x\|^2 - \|x_{k+1} - x\|^2\Big) + \frac{1}{2\rho_k}\Big(\|\lambda_k-\lambda\|^2 - \| \lambda_{k+1}-\lambda\|^2 \Big) +\frac{L_{F,\theta}}{2} \|\theta_{k+1} - \theta_k\|^2\notag \\
    &\qquad + \left(\frac{\rho_k C_1 +\|\lambda^*\|C_2+2 L_{F,x}+ L_{F,\theta}}{2}  - \frac{1}{2\gamma_k} \right) \|x_{k+1} - x_k\|^2 +\frac{C_2}{2} \|\lambda_k-\lambda^*\| \|x_{k+1}-x_k\|^2 \notag\\
&\qquad + \langle r_{k+1}, x_{k+1}-x\rangle - \langle r_k, x_k-x\rangle  + L_{F,\theta}\| \theta_{k+1}-\theta^*\| \|x_{k+1}-x\| +\frac{J \rho_k}{2} L^2_{f,\theta} \|\theta_k - \theta^*\|^2 \notag \\
&\qquad  + L^{\Phi}_{x \theta}\|\theta_k - \theta^*\| \Big(\|x_{k+1}-x_k\| + \|x_k - x\| \Big)  + \frac{L_{F,x}}{2} \left( \|x_k - x_{k-1}\|^2 - \|x_{k+1} - x_k\|^2 \right) \notag \\
&\qquad  + \frac{L_{F,\theta}}{2} \left( \|\theta_k - \theta_{k-1}\|^2 - \|\theta_{k+1} - \theta_k\|^2 \right) + L_{\lambda \theta}^{\Phi}\|\lambda_{k+1}-\lambda\|\|\theta_k-\theta^*\|  \notag\\
& \qquad+  2L_{f,\theta}(J\rho_k D_f+\sqrt{J}\|\lambda_k\|)\|\theta_k-\theta^*\|,
\end{align*}
where $r_k = F(x_k, \theta_k) - F(x_{k-1}, \theta_{k-1})$, $\gamma_k > 0$ is the step-size, and $\rho_k > 0$ is the penalty parameter.
\end{lemma}
\begin{proof}{Proof}
See Appendix \ref{SecB} for the proof.
\end{proof}

Note that the last term on the right-hand side of the inequality contains the product $\|\lambda_{k+1}-\lambda\|\|\theta_k-\theta^*\|$. Assuming that the learning iterates $\{\theta_k\}$ are converging to $\theta^*$, this term will be summable only if the dual iterates do not grow. In fact, in the following lemma, we establish that the dual iterates $\{\lambda_k\}_{k\geq 0}$ are bounded.

\begin{lemma} [Boundedness of Dual Iterates]\label{prop:prop2}
Suppose Assumption \ref{assum:assum1} holds, and define $D_{X}\triangleq\sup_{x\in X}\|x-x^*\| <\infty$ and $D_{\Theta} \triangleq \sup_{\theta\in\Theta}\|\theta-\theta^*\| <\infty$. Consider Algorithm \ref{alg:alm_mvi} with constant penalty $\rho_k \equiv \rho > 0$
and primal step-size
$\gamma_k \equiv \gamma \in \bigl(0,(\rho C_1+\|\lambda^*\|C_2+2L_{F,x}+L_{F,\theta}+C_2 D_\Lambda)^{-1}\bigr)$ for some $D_\Lambda>0$ where $C_1,C_2$ are defined in Lemma~\ref{lemma:lip_grad}.
Assume that the learning error satisfies $\sum_{k=0}^\infty
\|\theta_k-\theta^*\|
<
\infty$ and $\sum_{k=0}^\infty
(k+1)^2\|\theta_k-\theta^*\|^2
<
\infty,$ and $\rho_k=\rho=\frac{1}{L_{\lambda \theta}^\Phi}$ %$0<\rho L_{\lambda\theta}^{\Phi}\alpha_k<1$ 
for all $k\geq 0$. Then the dual sequence $\{\lambda_k\}$ is bounded. 
\end{lemma}

\begin{proof}{Proof}
  See Appendix \ref{SecC} for the proof.
\end{proof}

Lemma \ref{prop:prop2} establishes bounded dual iterates under compactness and summable learning errors when we use a constant primal step-size and a penalty parameter. We adopt this schedule henceforth, which yields a uniform bound $D_\Lambda\triangleq \sup_{k\geq 0}\|\lambda_k-\lambda^*\|<+\infty$. With this standing in place, we now state the main convergence result for the ergodic primal averages.

\begin{theorem} [Convergence Rate]\label{Thm:thm1} 
Suppose Assumption \ref{assum:assum1} holds, and Algorithm \ref{alg:alm_mvi} is executed under the conditions of Lemma \ref{prop:prop2}. Then, for the ergodic average $\bar x_K=\tfrac{1}{K}\sum_{k=1}^K x_k$, the following hold:
\begin{itemize}
    \item[(i)]  (Infeasibility:) $\mathbf{1}^\top [f(\bar{x}_K,\theta^*)]_+
\leq \frac{C_{\mathrm{feas}}}{K}$.
    \item [(ii)] (Relaxed gap:) $ \widetilde{\mathrm{Gap}}(\bar x_K,\theta^*) \triangleq
\sup_{x\in\mathcal{X}_{\epsilon_K}(\theta^*)}
F(x,\theta^*)^\top(\bar x_K-x)
\leq \frac{C'_T + C'_R(x,0)}{K}
+ D_\Lambda \frac{C_{\mathrm{feas}}}{K}
+ \frac{\rho}{2}\left(\frac{C_{\mathrm{feas}}}{K}\right)^2,$
\end{itemize}
where $C'_T \triangleq \tfrac{1}{2\gamma} D_X^2 + 2L_{F,\theta}D_XD_\Theta$,  $C'_R(x,0)= \Big(L_{F,\theta}D_X + 2L^\Phi_{x\theta}D_X + 2L^\Phi_{\lambda\theta}D_\Lambda
+2L_{f,\theta}\big(J\rho D_f+\sqrt{J}(D_\Lambda+\|\lambda^*\|)\big)\Big)\sum_{k=0}^{\infty}\|\theta_k-\theta^*\|
+\left(L_{F,\theta}+\frac{J\rho}{2}L_{f,\theta}^2\right)\sum_{k=0}^{\infty}\|\theta_k-\theta^*\|^2$, and $C_{\mathrm{feas}} = C'_T +\Big(L_{F,\theta}D_X + 2L^\Phi_{x\theta}D_X + 2 L^\Phi_{\lambda\theta}D_\Lambda
+2L_{f,\theta}\big(J\rho D_f+\sqrt{J}(D_\Lambda+\|\lambda^*\|)\big)\Big)
\sum_{k=0}^{\infty}\|\theta_k-\theta^*\| 
+\left(L_{F,\theta}+\frac{J\rho}{2}L_{f,\theta}^2\right)
\sum_{k=0}^{\infty}\|\theta_k-\theta^*\|^2
+\frac{J+\|\lambda_0-\lambda^*\|^2}{\rho}$.
\end{theorem}

\begin{proof}{Proof}
  See Appendix \ref{SecD} for the proof. 
\end{proof}

\begin{remark}
Based on Assumption \ref{assum:assum1}-(ii), the secondary VI problem \eqref{eq:SVI} has a Lipschitz continuous and strongly monotone operator which implies that the sequence $\{\theta_k\}_{k\geq 0}$ generated by the projected gradient step in line 7 of Algorithm \ref{alg:alm_mvi} converges to the optimal unique parameter $\theta^*$ at a linear convergence rate (e.g. see \citep[Proposition 26.16]{bauschke2011convex}); hence, the learning error accumulations $\sum_k \|\theta_k-\theta^*\|$ and $\sum_k (k+2)^2\|\theta_k-\theta^*\|^2$ in the upper bounds of the result of Theorem \ref{Thm:thm1} are finite with a uniform bound. Therefore, the result of Theorem \ref{Thm:thm1} implies that achieving $\epsilon$-infeasibility and $\epsilon$-relaxed gap requires at most running $\mathcal O(1/\epsilon)$ iterations of Algorithm \ref{alg:alm_mvi}. We finally remark that the update of misspecified parameter $\theta$ in the algorithm can be replaced with other methods, such as extra-gradient method, as long as $\theta_k\to \theta^*$ at a linear rate.
\end{remark}

\section{Numerical Experiment}\label{sec:numerics}
In this section, we test the performance of our proposed method by solving the Cournot model described in Section \ref{sec:motivating-example} and compare with the Tikhonov regularized VI scheme \citep{jiang2011learning,ahmadi2020resolution} and Extragradient method by \citep{ahmadi2020resolution}.

We consider a market with $N$ firms, where each firm $ i \in \{1,\dots,N\} $ produces $D$ homogeneous products so the total decision variable dimension is $n = ND$. For each product $d = 1,\dots,D$ of firm $i$, the capacity is $\text{Cap}_i=5$ and the quadratic cost function $c_i(x_i)=\tfrac{1}{2} r_i x_i^2 + g_i x_i$ where $r_i,g_i$ are randomly generated from uniform distributions $U[1,10]$ and $U[5,20]$, respectively. Moreover, $p(X; b^*) = a^* - b^* X$ where $X=\sum_{i=1}^N x_i$, we let $a^*=100$, and $b^*$ is the solution of the least-square minimization \eqref{eq:LS} with synthetic observations generated from a true model $X_t \sim U[2,20]$, $p_{\mathrm{obs},t} = a^* - b^* X_t$, for $t=1,\dots,T=300$. All the firms should also satisfy the coupling constraint of the form $f(\mathbf{x},b^*) \triangleq p(X;b^*) -\delta \leq 0$ for some $\delta>0$. 

We implemented our proposed ALM-Misspecified VI method and compare it with the misspecified Tikhonov regularized VI \citep{jiang2011learning,ahmadi2020resolution} and Extragradient (EG) method by \citep{ahmadi2020resolution}. For all methods, the misspecified parameter is updated using one projected gradient step at each iteration. To handle the coupling constraint for the other two methods, we employ the Lagrangian reformulation for each player and construct the corresponding VI operator by concatenating all the players' gradient vectors of the Lagrangian function with respect to primal and dual variables. We consider three problem instances with $(N,D)\in\{(50,5),(50,10),(100,10)\}$ and run each method for a sufficiently large number of iterations with constant step-sizes chosen according to the theoretical findings. Figure \ref{fig:cournot_all}, shows the performance of the methods in terms of the VI gap $\widetilde{\mathrm{Gap}}(\bar{x}_K,b^*)$ and infeasibility $\mathbf{1}^\top [f(\bar{x}_K,b^*)]_+$. Our proposed method (ALM-Misspecified VI) exhibits a superior performance across all instances, while the Lagrangian-Tikhonov and EG-Lagrangian baselines converge more slowly and deteriorate more noticeably as $(N,D)$ grow.
\begin{figure}[H]
    \centering
    \includegraphics[
        width=\textwidth,
        height=0.40\textheight,
        keepaspectratio
    ]{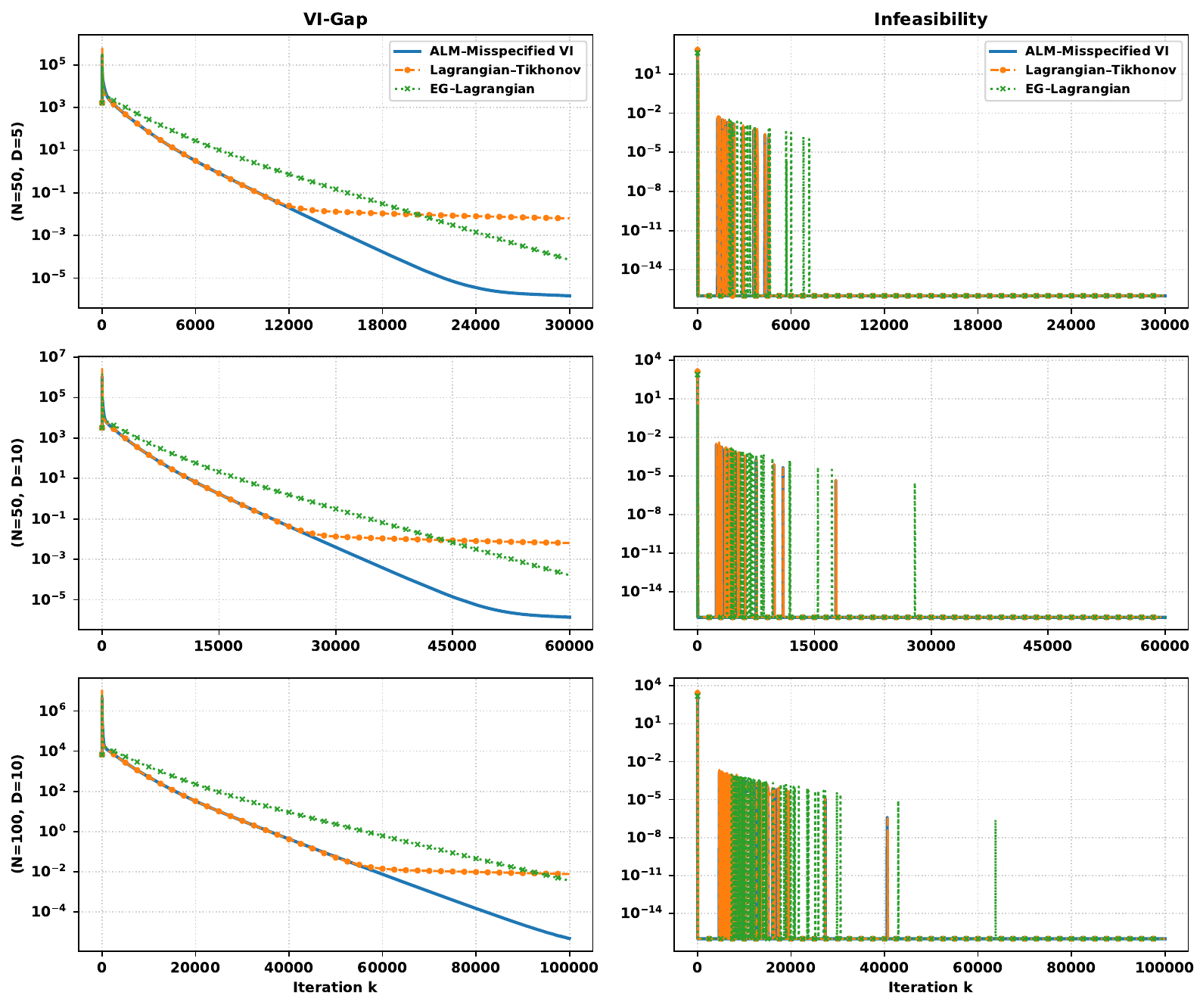}
    \caption{Cournot instances with $(N,D)\in\{(50,5),(50,10),(100,10)\}$.
    }
    \label{fig:cournot_all}
\end{figure}
The plots show the VI-Gap (left) and infeasibility (right) for ALM-Misspecified VI (solid), Lagrangian--Tikhonov (dashed with circles), and EG--Lagrangian (dotted with crosses).

\bibliographystyle{unsrtnat}
\bibliography{sample}

\newpage\begin{appendices}{Technical Proofs }%\begin{APPENDICES}
We include here the full proofs of the lemmas and the theorem for completeness.
\section{Proof of Lemma \ref{lemma:lemma2}}
\label{sec:lem2}
\begin{proof}{Proof}
Let $j$ be an index in $\{1,\dots,J\}$. For any $j^{th}$ coordinate using Algorithm \ref{alg:alm_mvi} (line 6), the dual update is $\displaystyle
    \lambda_{k+1}^{(j)}= \Bigl[\lambda_k^{(j)} + \rho_k f_j(x_{k+1},\theta_k)\Bigr]_+,$
where $[\cdot]_+$ denotes the projection onto $\mathbb{R}_+$. We now consider two cases, when $f_j's$ are active, we have $\displaystyle\lambda_k^{(j)} + \rho_k f_j(x_{k+1},\theta_k) > 0.$ Then the projection acts as the identity, so we have
\begin{align*}
\lambda_{k+1}^{(j)} &= \lambda_k^{(j)} + \rho_k f_j(x_{k+1},\theta_k) \quad \implies
\lambda_{k+1}^{(j)} - \lambda_k^{(j)}= \rho_k f_j(x_{k+1},\theta_k).
\end{align*}
Recall the definition \ref{def:def_aug_lag}, then we have
\begin{align}\label{eq:lemma5.1}
\Bigl[\nabla_{\lambda}\Phi_{\rho_k}(x_{k+1},\lambda_k,\theta_k)\Bigr]_j =
\begin{cases}f_j(x_{k+1},\theta_k) & \text{if } \rho_k f_j(x_{k+1},\theta_k)+\lambda_k^{(j)}\ge 0\\
-\dfrac{\lambda_k^{(j)}}{\rho_k} & \text{if } \rho_k f_j(x_{k+1},\theta_k)+\lambda_k^{(j)}< 0.
\end{cases}
\end{align}
Thus, we conclude that in the active case $\lambda_{k+1}^{(j)} - \lambda_k^{(j)} = \rho_k\Bigl[\nabla_{\lambda}\Phi_{\rho_k}(x_{k+1},\lambda_k,\theta_k)\Bigr]_j.$ Again, if $f_j(x_{k+1},\theta_k)'s$ are inactive, then we have $\lambda_k^{(j)} + \rho_k f_j(x_{k+1},\theta_k) < 0$, hence, $f_j(x_{k+1},\theta_k) < - \frac{\lambda_k^{(j)}}{\rho_k}.$ Thus, we have $\rho_k\Bigl[\nabla_{\lambda}\Phi_{\rho_k}(x_{k+1},\lambda_k,\theta_k)\Bigr]_j = -\lambda_k^{(j)}.$ Therefore, in inactive case also we have $\lambda_{k+1}^{(j)} - \lambda_k^{(j)} = \rho_k\Bigl[\nabla_{\lambda}\Phi_{\rho_k}(x_{k+1},\lambda_k,\theta_k)\Bigr]_j.$ Since the above equality holds for both active and inactive cases, we can write in full vector notation that $\displaystyle \lambda_{k+1} - \lambda_k = \rho_k\nabla_{\lambda}\Phi_{\rho_k}(x_{k+1},\lambda_k,\theta_k).$ Now by adding and subtracting $\nabla_{\lambda}\Phi_{\rho_k}(x_{k+1},\lambda_k,\theta^*)$, we get
\begin{align}\label{eq:lamdak1}
  \lambda_{k+1}-\lambda_k= \rho_k\Bigl[e_k+\nabla_{\lambda}\Phi_{\rho_k}\bigl(x_{k+1},\lambda_k,\theta^*\bigr)
  \Bigr],
\end{align}
where in the last equality, we define $e_k =\nabla_{\lambda}\Phi_{\rho_k}\bigl(x_{k+1},\lambda_k,\theta_k\bigr)-\nabla_{\lambda}\Phi_{\rho_k}\bigl(x_{k+1},\lambda_k,\theta^*\bigr)$. For any $\lambda\in \mathbb{R}_+^J$, using the three-point identity and dividing both sides by $\rho_k$ we get
\begin{align}\label{eq:lemma5:eq1}
  \frac{1}{\rho_k} (\lambda_{k+1} - \lambda)^\top(\lambda_{k+1}-\lambda_k)
   &= 
  \frac{1}{2 \rho_k} \Bigl[ 
    \|\lambda_{k+1}-\lambda\|^2
     - 
    \|\lambda_k - \lambda\|^2
     + 
    \|\lambda_{k+1}-\lambda_k\|^2
  \Bigr].
\end{align}
Using the expression in \eqref{eq:lamdak1} for $\lambda_{k+1}-\lambda_k$, we see
\begin{align}\label{eq:lemma5:eq2}
 \frac{1}{\rho_k} (\lambda_{k+1} - \lambda)^\top(\lambda_{k+1}-\lambda_k)
   &= 
  (\lambda_{k+1} - \lambda)^\top\nabla_{\lambda}\Phi_{\rho_k}\bigl(x_{k+1},\lambda_k,\theta^*\bigr)   + 
  (\lambda_{k+1} - \lambda)^\top e_k. 
\end{align}
Notice that, using (\ref{eq:lemma5:eq1}), we can rewrite (\ref{eq:lemma5:eq2}) as follows
\begin{align}\label{eq:lemma5:eq3}
  \frac{1}{2 \rho_k} \|\lambda_{k+1}-\lambda_k\|^2 &= 
  \frac{1}{2 \rho_k} \|\lambda_k - \lambda\|^2-  \frac{1}{2 \rho_k} \|\lambda_{k+1}-\lambda\|^2 + 
  (\lambda_{k+1} - \lambda)^\top\nabla_{\lambda}\Phi_{\rho_k}\bigl(x_{k+1},\lambda_k,\theta^*\bigr) +
  (\lambda_{k+1} - \lambda)^\top e_k \notag\\
  & \leq \frac{1}{2 \rho_k}\left( \|\lambda_k - \lambda\|^2- \|\lambda_{k+1}-\lambda\|^2\right)+  (\lambda_{k+1} - \lambda)^\top\nabla_{\lambda}\Phi_{\rho_k}\bigl(x_{k+1},\lambda_k,\theta^*\bigr) \notag\\
&\quad +L^\Phi_{\lambda\theta}\|\lambda_{k+1}-\lambda\|\|\theta_k-\theta^*\|,
\end{align}
where the last inequality follows from the fact that $\nabla_{\lambda}\Phi_{\rho_k}(x,\lambda,\theta)$ is Lipschitz continuous in $\theta$ with constant $L^\Phi_{\lambda\theta}$. Now let us define the following sets $J^+_{k} \triangleq \Big\{ j\in[J] : \rho_k f_j(x_{k+1},\theta_k) + \lambda_k^{(j)} \ge 0 \Big\} \text{ and }
J^-_{k} \triangleq [J]\setminus J^+_{k}$. For any  $ j \in J_k^+ $, the dual update in Algorithm \ref{alg:alm_mvi} (line 6) gives
\begin{align*}
    \lambda_{k+1}^{(j)} - \lambda_k^{(j)} &= \rho_k f_j(x_{k+1}, \theta_k) = \rho_k \left[ f_j(x_{k+1}, \theta^*) +\left(f_j(x_{k+1}, \theta_k) - f_j(x_{k+1}, \theta^*)\right)\right].
\end{align*}
Since $a=b+c$, squaring both sides and using the fact that $(b+c)^2 \geq \frac{1}{2} b^2- c^2$, we obtain
\begin{align*}
    \left(\lambda_{k+1}^{(j)}-\lambda_k^{(j)}\right)^2
&=\rho_k^2\left(f_j(x_{k+1},\theta_k)\right)^2 \\
&\geq
\rho_k^2\left(
\frac{1}{2} \left(f_j(x_{k+1},\theta^*)\right)^2
-
\left(f_j(x_{k+1},\theta_k)-f_j(x_{k+1},\theta^*)\right)^2
\right).
\end{align*}
Finally, invoking Lipschitz continuity of  $f_j(x, \theta)$ in $\theta$ from Assumption \ref{assum:assum1}(iii), we have
\begin{align*}
    \left(\lambda_{k+1}^{(j)} - \lambda_k^{(j)}\right)^2 \geq \rho_k^2 \left[\frac{1}{2}\left(f_j(x_{k+1}, \theta^*)\right)^2 - L^2_{f,\theta} \|\theta_k - \theta^*\|^2 \right].
\end{align*}
Summing over $ j \in J_k^+ $ and multiplying by $1/(2\rho_k)$, we have
\begin{align*}
\frac{1}{2\rho_k}\sum_{j \in J_k^+} \left(\lambda_{k+1}^{(j)} - \lambda_k^{(j)}\right)^2 \geq \frac{\rho_k}{4} \sum_{j \in J_k^+} \left(f_j(x_{k+1}, \theta^*)\right)^2 - \frac{J\rho_k}{2} L^2_{f,\theta} \|\theta_k - \theta^*\|^2.
\end{align*}
If $j\in J_k^-$, then $\lambda^{(j)}_{k+1}=0 \implies \left(\lambda^{(j)}_{k+1}-\lambda^{(j)}_k \right)^2=\left(\lambda^{(j)}_k \right)^2$. Multiplying both sides by $\frac{1}{2\rho_k}$, and combining both of the above inequalities, we get
\begin{align*}
\frac{1}{2\rho_k}\|\lambda_{k+1}-\lambda_k\|^2
  \geq \tfrac{\rho_k}{4} \sum_{j\in J_k^+}\left(f_j(x_{k+1},\theta^*)\right)^2
  +\tfrac1{2\rho_k}\sum_{j\in J_k^-}(\lambda_k^{(j)})^2  -\frac{J\rho_k}{2}  L^2_{f,\theta} \|\theta_k - \theta^*\|^2.
\end{align*}
After rearranging, this gives us
\begin{align}\label{eq:eq_dual_1}
 & - \left(\frac{1}{2\rho_k}\|\lambda_{k+1}-\lambda_k\|^2
   +\frac{J \rho_k}{2} L^2_{f,\theta} \|\theta_k - \theta^*\|^2 \right)  \leq - \Biggl[ \tfrac{\rho_k}{4} \sum_{j\in J_k^+} \left(f_j(x_{k+1},\theta^*)\right)^2+\tfrac1{2\rho_k}\sum_{j\in J_k^-}(\lambda_k^{(j)})^2\Biggr].
\end{align}
Now, notice that
\begin{align*}
\Phi_{\rho_k}\bigl(x_{k+1},\lambda_k,\theta^*\bigr) & = \Phi_{\rho_k}\bigl(x_{k+1},\lambda_k,\theta^*\bigr)- \Phi_{\rho_k}\bigl(x_{k+1},\lambda_k,\theta_k\bigr)+\Phi_{\rho_k}\bigl(x_{k+1},\lambda_k,\theta_k\bigr) \notag\\
&=\sum_{j\in J_k^+}
        \Bigl(\frac{\rho_k}{2}\bigl(f_j(x_{k+1},\theta_k)\bigr)^2
              + \lambda_k^{(j)} f_j(x_{k+1},\theta_k)\Bigr)-
      \sum_{j\in J_k^-}\frac{\bigl(\lambda_k^{(j)}\bigr)^2}{2\rho_k} \notag\\
      & \qquad +\Phi_{\rho_k}\bigl(x_{k+1},\lambda_k,\theta^*\bigr)- \Phi_{\rho_k}\bigl(x_{k+1},\lambda_k,\theta_k\bigr).
\end{align*}
For each $j \in J_k^+$, we add and subtract the term $\tfrac{\rho_k}{2}\bigl(f_j(x_{k+1},\theta^*)\bigr)^2+\lambda_k^{(j)}f_j(x_{k+1},\theta^*)$ inside the summation, hence we obtain
\begin{align}\label{eq:lemma5:eq5}
    \Phi_{\rho_k}(x_{k+1},\lambda_k,\theta^*)
&= \sum_{j\in J_k^+}
\Bigl(
\tfrac{\rho_k}{2}\left(f_j(x_{k+1},\theta^*)\right)^2
+
\lambda_k^{(j)}f_j(x_{k+1},\theta^*)
\Bigr) -\sum_{j\in J_k^-}\frac{(\lambda_k^{(j)})^2}{2\rho_k}+ E_k,
\end{align}
where $E_k\triangleq 
\sum_{j\in J_k^+}
\left(
\tfrac{\rho_k}{2}\bigl(\left(f_j(x_{k+1},\theta_k)\right)^2-\left(f_j(x_{k+1},\theta^*)\right)^2\bigr)
+
\lambda_k^{(j)}\bigl(f_j(x_{k+1},\theta_k)-
f_j(x_{k+1},\theta^*)\bigr)\right)+\Phi_{\rho_k}\bigl(x_{k+1},\lambda_k,\theta^*\bigr)- \Phi_{\rho_k}\bigl(x_{k+1},\lambda_k,\theta_k\bigr)$. Now using the facts from \eqref{eq:lemma5.1} and \eqref{eq:lemma5:eq5}, we can expand as follows
\begin{align*}
&\Phi_{\rho_k}\bigl(x_{k+1},\lambda_k,\theta^*\bigr)
 - \sum_{j=1}^J\lambda^{(j)}f_j(x_{k+1},\theta^*)
 - \left(\lambda_{k+1}-\lambda\right)^\top 
   \nabla_{\lambda}\Phi_{\rho_k}(x_{k+1},\lambda_k,\theta^*) \\
&\quad= \sum_{j\in J_k^+}
        \Bigl(\tfrac{\rho_k}{2}\bigl(f_j(x_{k+1},\theta^*)\bigr)^2
              + \lambda_k^{(j)} f_j(x_{k+1},\theta^*)\Bigr)
       - \sum_{j\in J_k^-}\tfrac{(\lambda_k^{(j)})^2}{2\rho_k} -\sum_{j\in J_k^+}
      \bigl(\lambda^{(j)}_{k+1}-\lambda^{(j)} \bigr)f_j(x_{k+1},\theta^*)\\
&\qquad- \sum_{j\in J_k^-}
     \bigl(\lambda^{(j)}_{k+1}-\lambda^{(j)} \bigr)\Bigl(-\tfrac{\lambda^{(j)}_k}{\rho_k} \Bigr) 
     - \sum_{j\in J_k^+}\lambda^{(j)}f_j(x_{k+1},\theta^*)
     - \sum_{j\in J_k^-}\lambda^{(j)}f_j(x_{k+1},\theta^*) +E_k.
\end{align*}
Next, by grouping the positive and negative index sets on the RHS separately, we get
\begin{align*}
       &=  \sum_{j\in J_k^+}
        \Bigl(\frac{\rho_k}{2}\bigl(f_j(x_{k+1},\theta^*)\bigr)^2
              + \lambda_k^{(j)} f_j(x_{k+1},\theta^*)\Bigr)-
      \sum_{j\in J_k^-}\frac{\bigl(\lambda_k^{(j)}\bigr)^2}{2\rho_k}- \sum_{j\in J_k^-}\lambda^{(j)}f_j(x_{k+1},\theta^*)\notag\\
       & \quad- \sum_{j\in J_k^+}\left(\lambda^{(j)}_k+\rho_k f_j(x_{k+1},\theta^*)-\lambda^{(j)} \right) f_j(x_{k+1},\theta^*) - \sum_{j\in J_k^-}\left(-\lambda^{(j)}\right)\left(-\frac{\lambda^{(j)}_k}{\rho_k} \right) -  \sum_{j\in J^+_k}\lambda^{(j)}f_j(x_{k+1},\theta^*)+ E_k.
\end{align*}
Now, observe that the terms inside the sum over $J_k^+$ can be combined, since both involve $f_j(x_{k+1},\theta^*)$. After cancellation we arrive at
\begin{align*}
          &= \sum_{j\in J_k^+}\Bigl[\tfrac{\rho_k}{2}\bigl(f_j(x_{k+1},\theta^*)\bigr)^2
              + \lambda_k^{(j)} f_j(x_{k+1},\theta^*)-\lambda^{(j)}f_j(x_{k+1},\theta^*) \notag\\
              &\quad -\left(\lambda^{(j)}_k+\rho_k f_j(x_{k+1},\theta^*)-\lambda^{(j)} \right)f_j(x_{k+1},\theta^*) \Bigr]  - \sum_{j\in J_k^-}\left(\tfrac{\bigl(\lambda_k^{(j)}\bigr)^2}{2\rho_k}+\lambda^{(j)}\Bigl(f_j(x_{k+1},\theta^*)+\tfrac{\lambda_k^{(j)}}{\rho_k}\Bigr)\right)+E_k.
\end{align*}
Finally, simplifying the quadratic terms in $f_j(x_{k+1},\theta^*)$, we can conclude that
\begin{align}\label{eq:lemma5.2}
&\Phi_{\rho_k}\bigl(x_{k+1},\lambda_k,\theta^*\bigr)
 - \sum_{j=1}^J\lambda^{(j)}f_j(x_{k+1},\theta^*)
 - \left(\lambda_{k+1}-\lambda\right)^\top 
   \nabla_{\lambda}\Phi_{\rho_k}(x_{k+1},\lambda_k,\theta^*) \notag\\
   & \quad  = - \sum_{j\in J_k^+}\tfrac{\rho_k}{2}\bigl(f_j(x_{k+1},\theta^*)\bigr)^2  - \sum_{j\in J_k^-}\left(\frac{\bigl(\lambda_k^{(j)}\bigr)^2}{2\rho_k}+\lambda^{(j)}\Bigl(f_j(x_{k+1},\theta^*)+\frac{\lambda_k^{(j)}}{\rho_k}\Bigr)\right) +E_k.
\end{align}
Note that, for any $\lambda^{(j)}\geq 0$ and $\forall j\in J_k^-$, we have $\Bigl(f_j(x_{k+1},\theta^*)+\frac{\lambda_k^{(j)}}{\rho_k}\Bigr)\leq 0$. Then from \eqref{eq:lemma5.2}, we have
\begin{align}\label{eq:lemma5.3}
&\Phi_{\rho_k}\bigl(x_{k+1},\lambda_k,\theta^*\bigr)- \sum_{j=1}^J\lambda^{(j)}f_j(x_{k+1},\theta^*)-\left(\lambda_{k+1}-\lambda\right)^\top \nabla_{\lambda}\Phi_{\rho_k}(x_{k+1},\lambda_k,\theta^*)\notag\\
     &\quad \geq -\left(\sum_{j\in J_k^+}\frac{\rho_k}{2}\bigl(f_j(x_{k+1},\theta^*\bigr)^2 +\sum_{j\in J_k^-}\frac{\bigl(\lambda_k^{(j)}\bigr)^2}{2\rho_k}\right) +E_k\notag\\
    & \quad \geq -\left(\frac{1}{2\rho_k}\|\lambda_{k+1}-\lambda_k\|^2
   +\frac{J \rho_k}{2} L^2_{f,\theta} \|\theta_k - \theta^*\|^2 \right) +E_k,
\end{align}
where the last inequality follows from \eqref{eq:eq_dual_1}. Therefore, we have
\begin{align}\label{eq:lemma5:eq4}
    -&\Phi_{\rho_k}\bigl(x_{k+1},\lambda_k,\theta^*\bigr)+ \sum_{j=1}^J\lambda^{(j)}f_j(x_{k+1},\theta^*) +\left(\lambda_{k+1}-\lambda\right)^\top \nabla_{\lambda}\Phi_{\rho_k}(x_{k+1},\lambda_k,\theta^*) \nonumber\\
    & \quad \leq 
    \left(\frac{1}{2\rho_k}\|\lambda_{k+1}-\lambda_k\|^2
   +\frac{J \rho_k}{2} L^2_{f,\theta} \|\theta_k - \theta^*\|^2 \right) -E_k.
\end{align}
Recall that $E_k=
\sum_{j\in J_k^+}
\left(
\frac{\rho_k}{2}\Bigl(f_j(x_{k+1},\theta_k)^2-f_j(x_{k+1},\theta^*)^2\Bigr)
+ \lambda_k^{(j)}\Bigl(f_j(x_{k+1},\theta_k)-f_j(x_{k+1},\theta^*)\Bigr)
\right) +\Phi_{\rho_k}(x_{k+1},\lambda_k,\theta^*)-\Phi_{\rho_k}(x_{k+1},\lambda_k,\theta_k).$
Using the triangle inequality and $a^2-b^2=(a-b)(a+b)$, we obtain
\begin{align}\label{eq:Ek-bound}
|E_k| &\leq \sum_{j\in J_k^+}
\frac{\rho_k}{2}
|f_j(x_{k+1},\theta_k)-f_j(x_{k+1},\theta^*)|
|f_j(x_{k+1},\theta_k)+f_j(x_{k+1},\theta^*)| \nonumber\\
& \qquad +\sum_{j\in J_k^+}
|\lambda_k^{(j)}|
|f_j(x_{k+1},\theta_k)-f_j(x_{k+1},\theta^*)| +
\left|
\Phi_{\rho_k}(x_{k+1},\lambda_k,\theta^*)-\Phi_{\rho_k}(x_{k+1},\lambda_k,\theta_k)
\right|.
\end{align}

Recall that $\Phi_\rho(x,\lambda,\theta)=\sum_{j=1}^J \phi_\rho(f_j(x,\theta),\lambda_j)$, moreover, one can observe that for any $u_j,u_j'\in\mathbb R$, $|\phi_\rho(u_j,v_j)-\phi_\rho(u_j',v_j)|
\le (|v_j|+\rho D_f)|u_j-u_j'|$. Therefore, by Lipschitz continuity of $f_j(x,\theta)$ in $\theta$ we conclude that for any $\theta,\theta'\in\Theta$,  $|\phi_\rho(f_j(x,\theta),\lambda^{(j)})-\phi_\rho(f_j(x,\theta'),\lambda^{(j)})|
\leq (|\lambda^{(j)}|+\rho D_f)L_{f,\theta}\|\theta-\theta'\|$ which summing over $j$ implies that 
\begin{align}\label{eq:lip-Phi-theta}
    |\Phi_\rho(x,\lambda,\theta)-\Phi_\rho(x,\lambda,\theta')|\leq
L_{f,\theta}(J\rho D_f+\sqrt{J}\|\lambda\|)\|\theta-\theta'\|.
\end{align}

Using the above inequality within \eqref{eq:Ek-bound} and by Lipschitz continuity of $f_j(x,\theta)$ in $\theta$ and the bound $|f_j(x,\theta)|\leq D_f$, we conclude that
\begin{align}\label{eq:lemma5:eq2.1a}
    |E_k| \leq 2L_{f,\theta}(J\rho_k D_f+\sqrt{J}\|\lambda_k\|)\|\theta_k-\theta^*\|.
\end{align}
Now using \eqref{eq:lemma5:eq3} and \eqref{eq:lemma5:eq2.1a} in \eqref{eq:lemma5:eq4}, we can show
\begin{align}\label{eq:lemma5:eq5b}
-\Phi_{\rho_k}\bigl(x_{k+1},\lambda_k,\theta^*\bigr)
+  \sum_{j=1}^J\lambda^{(j)}f_j(x_{k+1},\theta^*) &\leq
\frac{1}{2 \rho_k}\left( \|\lambda_k - \lambda\|^2 - \|\lambda_{k+1}-\lambda\|^2\right) + L^\Phi_{\lambda\theta}\|\lambda_{k+1}-\lambda\| \|\theta_k-\theta^*\| \notag\\
& \quad +\frac{J \rho_k}{2} L^2_{f,\theta} \|\theta_k - \theta^*\|^2 + 2L_{f,\theta}(J\rho_k D_f+\sqrt{J}\|\lambda_k\|)\|\theta_k-\theta^*\|.
\end{align}
\end{proof}

\section{Proof of Lemma \ref{lemma:lip_grad}}\label{sec:lem3}
\begin{proof}{Proof} Recall from Definition \ref{def:def_aug_lag} and \eqref{rem:rem6_prop2}, for any $x\in X$, we have
\begin{align*}
    \nabla_x \Phi_{\rho}(x,\lambda,\theta)
=
\sum_{j=1}^J [\rho f_j(x,\theta)+\lambda^{(j)}]_+\,\nabla_x f_j(x,\theta)
=
\nabla_x f(x,\theta)^\top[\rho f(x,\theta)+\lambda]_+.
\end{align*}
Fix $\rho>0$, $\lambda\in\mathbb R^J$, and $\theta\in\Theta$. Then for any $x,y\in X$, we have
\begin{align*}
 \nabla_x \Phi_{\rho}(x,\lambda,\theta)-\nabla_x \Phi_{\rho}(y,\lambda,\theta)
&= \nabla_x f(x,\theta)^\top[\rho f(x,\theta)+\lambda]_+
-\nabla_x f(y,\theta)^\top[\rho f(y,\theta)+\lambda]_+ \\
&= \big(\nabla_x f(x,\theta)-\nabla_x f(y,\theta)\big)^\top[\rho f(x,\theta)+\lambda]_+ \\
&\quad+ \nabla_x f(y,\theta)^\top\Big([\rho f(x,\theta)+\lambda]_+-[\rho f(y,\theta)+\lambda]_+\Big),
\end{align*}
where we added and subtracted $\nabla_x f(y,\theta)^\top[\rho f(x,\theta)+\lambda]_+$.
Taking norms and applying the triangle inequality yields
\begin{align}\label{eq:grad_lem1}
    \|\nabla_x \Phi_{\rho}(x,\lambda,\theta)-\nabla_x \Phi_{\rho}(y,\lambda,\theta)\| &\leq
\|\nabla_x f(x,\theta)-\nabla_x f(y,\theta)\|\,
\|[\rho f(x,\theta)+\lambda]_+\| \notag\\
&\quad+
\|\nabla_x f(y,\theta)\|\,
\|[\rho f(x,\theta)+\lambda]_+-[\rho f(y,\theta)+\lambda]_+\|.
\end{align}
Since $[\cdot]_+$ is the projection onto $\mathbb{R}^J_+$, it is non-expansive. Therefore,
\begin{align*}
\|f(x,\theta)-f(y,\theta)\|
= \Big(\sum_{j=1}^J |f_j(x,\theta)-f_j(y,\theta)|^2\Big)^{1/2}
\leq \sqrt{J}\,L_{f,x}\|x-y\|.
\end{align*}
Consequently, it follows that
\begin{align*}
    \|[\rho f(x,\theta)+\lambda]_+-[\rho f(y,\theta)+\lambda]_+\| \leq
\rho\sqrt{J}\,L_{f,x}\|x-y\|.
\end{align*}
Similarly, by Assumption \ref{assum:assum1}(iii), each $\nabla_x f_j(\cdot,\theta)$ is Lipschitz continuous in $x$ with constant $L_{\nabla f,x}$. Therefore, the Jacobian matrix satisfies
\begin{align*}
    \|\nabla_x f(x,\theta)-\nabla_x f(y,\theta)\| \leq
\sqrt{J}\,L_{\nabla f,x}\|x-y\|.
\end{align*}
Substituting these bounds into \eqref{eq:grad_lem1}, we obtain
\begin{align*}
    \|\nabla_x \Phi_{\rho}(x,\lambda,\theta)-\nabla_x \Phi_{\rho}(y,\lambda,\theta)\| &\leq
\sqrt{J}\,L_{\nabla f,x}\|x-y\|\,\|[\rho f(x,\theta)+\lambda]_+\| \\
&\quad + \rho\sqrt{J}\,L_{f,x}\|x-y\|\,\|\nabla_x f(y,\theta)\|.
\end{align*}
Now, since $X\times\Theta$ is compact and the mappings $(z,\theta)\mapsto f(z,\theta)$ and $(z,\theta)\mapsto \nabla_x f(z,\theta)$ are continuous, the quantities $D_f\triangleq \sup_{(z,\vartheta)\in X\times\Theta}\|[f(z,\vartheta)]_+\|<\infty,
\quad
M_{\nabla f}\triangleq \sup_{(z,\vartheta)\in X\times\Theta}\|\nabla_x f(z,\vartheta)\|<\infty,\ $ are well defined and finite. Moreover, $\|[\rho f(x,\theta)+\lambda]_+\|
\leq \rho\|[f(x,\theta)]_+\|+\|\lambda\|
\leq \rho D_f+\|\lambda\|,$ and $\|\nabla_x f(y,\theta)\|\leq M_{\nabla f}$. Hence, we get
\begin{align*}
    \|\nabla_x \Phi_{\rho}(x,\lambda,\theta)-\nabla_x \Phi_{\rho}(y,\lambda,\theta)\|
&\leq \sqrt{J}\,L_{\nabla f,x}(\rho D_f+\|\lambda\|)\|x-y\| %\\
%&\quad
+\rho\sqrt{J}\,L_{f,x}M_{\nabla f}\|x-y\|.
\end{align*}
Rearranging terms gives
\begin{align*}
    \|\nabla_x \Phi_{\rho}(x,\lambda,\theta)-\nabla_x \Phi_{\rho}(y,\lambda,\theta)\|
\leq \Big(\rho C_1+\|\lambda\|C_2\Big)\|x-y\|,
\end{align*}
where $C_1\triangleq \sqrt{J}\big(L_{\nabla f,x}D_f+L_{f,x}M_{\nabla f}\big),
\quad
C_2\triangleq \sqrt{J}\,L_{\nabla f,x}.$ Thus, $\nabla_x\Phi_\rho(\cdot,\lambda,\theta)$ is Lipschitz continuous on $X$. 
\end{proof}

\section{Proof of Lemma \ref{prop:prop_1}}\label{SecB}
\begin{proof}{Proof} Let $x\in X$ and $\lambda \in \mathbb{R}^J_+$ be arbitrary vectors. From Algorithm \ref{alg:alm_mvi} (line 5) we have:
\begin{align}\label{eq:primal-updt2}
    (x_{k+1}-x)^\top\Bigg( x_{k+1}-x_k + \gamma_k \Big(F(x_k,\theta_k) + r_k+ \mathbf{J} f(x_k,\theta_k) \Big[ \rho_kf(x_k,\theta_k)+\lambda_k\Big]_+ \Big)\Bigg) \leq0.
\end{align}
Considering \eqref{eq:primal-updt2}, to find a lower bound for $\left(x_{k+1}-x\right)^{\top} \Big(F\left(x_k,\theta_k\right) + r_k \Big)$, we add and subtract $(x_{k+1}-x)^\top F(x_{k+1},\theta_{k+1})$, obtaining:
\begin{align*}
    &\left(x_{k+1}-x\right)^{\top} \Big(F\left(x_k,\theta_k\right) + r_k \Big) \pm (x_{k+1}-x)^\top F(x_{k+1},\theta_{k+1})\\
    & \quad = \langle r_k, x_k-x\rangle - \langle r_{k+1}, x_{k+1}-x\rangle + \langle r_k, x_{k+1}-x_k\rangle + \langle F(x_{k+1},\theta_{k+1}),x_{k+1}-x\rangle\\
    & \quad =\left(x_{k+1}-x\right)^{\top} \Big(F(x_{k+1},\theta^*) + [F(x_{k+1},\theta_{k+1})-F(x_{k+1},\theta^*)]  \Big) + \langle r_k, x_k-x\rangle \\
    & \qquad - \langle r_{k+1}, x_{k+1}-x\rangle + \langle r_k, x_{k+1}-x_k\rangle\\
    & \quad \geq \left(x_{k+1}-x\right)^{\top} F(x,\theta^*) +  \left(x_{k+1}-x\right)^{\top}[F(x_{k+1},\theta_{k+1})-F(x_{k+1},\theta^*)] +  \langle r_k, x_k-x\rangle\\
    &\qquad - \langle r_{k+1}, x_{k+1}-x\rangle + \langle r_k, x_{k+1}-x_k\rangle,
\end{align*}
where in the last inequality, we use the monotonicity of $F(\cdot,\theta^*)$. Next, using Cauchy-Schwarz inequality and Lipschitz continuity of $F(x,\cdot)$ with respect to $\theta$, we can show:
\begin{align}\label{eq:prop-p1}
    &\left(x_{k+1}-x\right)^{\top} \left(F (x_k,\theta_k) +r_k \right) \nonumber\\
    & \quad \geq \left(x_{k+1}-x\right)^{\top} F(x,\theta^*) - \Big(L_{F,x}\|x_k - x_{k-1}\| + L_{F,\theta}\|\theta_k - \theta_{k-1} \| \Big)\|x_{k+1}-x_k\| \nonumber\\
    & \qquad - L_{F,\theta}\|\theta_{k+1} - \theta^* \| \|x_{k+1}-x\| + \langle r_k, x_k-x\rangle - \langle r_{k+1}, x_{k+1}-x\rangle
\end{align}
Taking into account \ref{rem:rem6_prop2}, by adding and subtracting $\nabla_x\Phi_{\rho_k}(x_k,\lambda_k,\theta^*)$, we can obtain the following.
\begin{align*}
    &(x_{k+1} - x)^\top \mathbf{J} f(x_k,\theta_k) \Big[\rho_k f(x_k,\theta_k)+\lambda_k \Big]_+\\
    & \quad = (x_{k+1}-x_k)^\top \nabla_x \Phi_{\rho_k}(x_k,\lambda_k,\theta_k) + (x_{k}-x)^\top \nabla_x \Phi_{\rho_k}(x_k,\lambda_k,\theta_k)\\
    & \quad = (x_{k+1}-x_k)^\top \Big(\nabla_x \Phi_{\rho_k}(x_k,\lambda_k,\theta^*) +[\nabla_x \Phi_{\rho_k}(x_k,\lambda_k,\theta_k) - \nabla_x \Phi_{\rho_k}(x_k,\lambda_k,\theta^*)] \Big) \\
    & \qquad + (x_{k}-x)^\top \Big(\nabla_x \Phi_{\rho_k}(x_k,\lambda_k,\theta^*) +[\nabla_x \Phi_{\rho_k}(x_k,\lambda_k,\theta_k) - \nabla_x \Phi_{\rho_k}(x_k,\lambda_k,\theta^*)] \Big)\\
    & \quad \geq  (x_{k+1}-x_k)^\top \nabla_x \Phi_{\rho_k}(x_k,\lambda_k,\theta^*) + (x_{k+1}-x_k)^\top[\nabla_x \Phi_{\rho_k}(x_k,\lambda_k,\theta_k) - \nabla_x \Phi_{\rho_k}(x_k,\lambda_k,\theta^*)]\\
    & \qquad + \Phi_{\rho_k}(x_k,\lambda_k,\theta^*) - \Phi_{\rho_k}(x,\lambda_k,\theta^*) + (x_k-x)^\top[\nabla_x \Phi_{\rho_k}(x_k,\lambda_k,\theta_k) - \nabla_x \Phi_{\rho_k}(x_k,\lambda_k,\theta^*)],
\end{align*}
where in the last inequality, we use the convexity of $\Phi_{\rho_k}(\cdot,\lambda_k,\theta^*)$. Next, using the Cauchy-Schwarz inequality and Lipschitz continuity of $\nabla_x\Phi_{\rho_k}(x,\lambda,\cdot)$ with respect $\theta$ with a constant $L^{\Phi}_{x\theta}$, we can show
\begin{align}\label{eq:prop-p2}
    &(x_{k+1} - x)^\top \mathbf{J} f(x_k,\theta_k) \Big[\rho_k f(x_k,\theta_k)+\lambda_k \Big]_+ \geq (x_{k+1}-x_k)^\top \nabla_x \Phi_{\rho_k}(x_k,\lambda_k,\theta^*)\nonumber\\
    & \quad  + \Phi_{\rho_k}(x_k,\lambda_k,\theta^*) - \Phi_{\rho_k}(x,\lambda_k,\theta^*) - L^{\Phi}_{x\theta}\|\theta_k - \theta^* \| \Big(\|x_{k+1}-x_k \| + \|x_k - x \| \Big).
\end{align}
Moreover, using the three-point equality, we can write:
\begin{align}\label{eq:prop-p3}
    (x_{k+1}-x)^\top(x_{k+1}-x_k) = \frac{1}{2}\Big(\|x_{k+1} -x \|^2 - \|x_{k} -x \|^2 + \|x_{k+1} -x_k \|^2\Big).
\end{align}
Using \eqref{eq:prop-p1}, \eqref{eq:prop-p2}, and \eqref{eq:prop-p3} in \eqref{eq:primal-updt2}, we have:
\begin{align}\label{eq:prop-p4}
    &\left(x_{k
    +1}-x\right)^{\top} F\left(x,\theta^*\right) + \Phi_{\rho_k}(x_k,\lambda_k,\theta^*) - \Phi_{\rho_k}(x,\lambda_k,\theta^*)\nonumber\\
    & \quad \leq \frac{1}{2\gamma_k}\Big(\|x_{k} -x \|^2 - \|x_{k+1} -x \|^2 - \|x_{k+1} -x_k \|^2\Big) - (x_{k+1}-x_k)^\top \nabla_x \Phi_{\rho_k}(x_k,\lambda_k,\theta^*)  \notag\\
    & \qquad  + L_{F,\theta}\| \theta_{k+1}-\theta^*\| \|x_{k+1}-x\| + L^{\Phi}_{x \theta}\|\theta_k - \theta^*\| \Big(\|x_{k+1}-x_k\| + \|x_k - x\| \Big)\nonumber\\
    & \qquad + \|x_{k+1}-x_k\| \Big( L_{F,x}\|x_k - x_{k-1}\| + L_{F,\theta}\|\theta_{k}-\theta_{k-1}\|\Big) + \langle r_{k+1}, x_{k+1}-x\rangle - \langle r_k, x_k-x\rangle.
\end{align}
Next, using the Lemma \ref{lemma:lip_grad}, we have $\Phi_{\rho_k}(x_{k+1},\lambda_k,\theta^*)\leq \Phi_{\rho_k}(x_{k},\lambda_k,\theta^*) + \langle \nabla_x \Phi_{\rho_k}(x_k,\lambda_k,\theta^*),x_{k+1}-x_k\rangle + \frac{\rho_kC_1+\|\lambda_k\|C_2}{2}\|x_{k+1}-x_k\|^2$. Substituting this estimate into \eqref{eq:prop-p4} yields
\begin{align}\label{eq:prop-p5}
    &\left(x_{k
    +1}-x\right)^{\top} F\left(x,\theta^*\right) + \Phi_{\rho_{k}}(x_{k+1},\lambda_k,\theta^*) - \Phi_{\rho_k}(x,\lambda_k,\theta^*)\nonumber\\
    & \quad \leq \frac{1}{2\gamma_k}\Big(\|x_{k} -x \|^2 - \|x_{k+1} -x \|^2 - \|x_{k+1} -x_k \|^2\Big) + \langle r_{k+1}, x_{k+1}-x\rangle - \langle r_k, x_k-x\rangle \nonumber\\
    & \qquad
    + L_{F,\theta}\| \theta_{k+1}-\theta^*\| \|x_{k+1}-x\| + L^{\Phi}_{x \theta}\|\theta_k - \theta^*\| \Big(\|x_{k+1}-x_k\| + \|x_k - x\| \Big) \nonumber\\
    & \qquad + \frac{\rho_kC_1+\|\lambda_k\|C_2}{2}\|x_{k+1}-x_k\|^2 + \|x_{k+1}-x_k\| \Big( L_{F,x}\|x_k - x_{k-1}\| + L_{F,\theta}\|\theta_{k}-\theta_{k-1}\|\Big).
\end{align}
To treat the term involving $\|\lambda_k\|$, we first add and subtract $\lambda^*$ inside the norm and use the triangle inequality to obtain
\begin{align}\label{eq:lemma4a}
\frac{\rho_kC_1+\|\lambda_k\|C_2}{2}\|x_{k+1}-x_k\|^2&= \left(\frac{\rho_k C_1+\|\lambda^*+(\lambda_k-\lambda^*)\|C_2}{2}\right)\|x_{k+1}-x_k\|^2 \notag\\
    & \leq \frac{\rho_k C_1+\|\lambda^*\|C_2}{2}\|x_{k+1}-x_k\|^2 + \frac{C_2}{2} \|\lambda_k-\lambda^*\| \|x_{k+1}-x_k\|^2.
\end{align}
Using \eqref{eq:lemma5:eq5b} from Lemma \eqref{lemma:lemma2} and substituting \eqref{eq:lemma4a} into \eqref{eq:prop-p5}, we obtain
\begin{align}\label{eq:prop-p6}
    &\left(x_{k+1}-x\right)^{\top} F\left(x,\theta^*\right) + \sum_{j=1}^J\lambda^{(j)}f_j(x_{k+1},\theta^*) - \Phi_{\rho_k}(x,\lambda_k,\theta^*)\nonumber\\
    & \quad \leq \frac{1}{2\gamma_k}\Big(\|x_{k} -x \|^2 - \|x_{k+1} -x \|^2 - \|x_{k+1} -x_k \|^2\Big) + \langle r_{k+1}, x_{k+1}-x\rangle - \langle r_k, x_k-x\rangle  +\frac{J \rho_k}{2} L^2_{f,\theta} \|\theta_k - \theta^*\|^2 \nonumber\\
    & \qquad + L_{F,\theta}\| \theta_{k+1}-\theta^*\| \|x_{k+1}-x\| + L^{\Phi}_{x \theta}\|\theta_k - \theta^*\| \Big(\|x_{k+1}-x_k\| + \|x_k - x\| \Big) +L^\Phi_{\lambda\theta}\|\lambda_{k+1}-\lambda\| \|\theta_k-\theta^*\|   \nonumber\\
    & \qquad +\|x_{k+1}-x_k\| \Big( L_{F,x}\|x_k - x_{k-1}\| + L_{F,\theta}\|\theta_{k}-\theta_{k-1}\|\Big) + \frac{1}{2\rho_k}\Big(\|\lambda_k-\lambda \|^2  - \| \lambda_{k+1}-\lambda\|^2 \Big)  \nonumber\\
    & \qquad  +  \frac{\rho_k C_1+ \|\lambda^*\|C_2}{2}\|x_{k+1}-x_k\|^2 + \frac{C_2}{2} \|\lambda_k-\lambda^*\| \|x_{k+1}-x_k\|^2 +  2L_{f,\theta}(J\rho_k D_f+\sqrt{J}\|\lambda_k\|)\|\theta_k-\theta^*\|.
\end{align}
Note that applying Young’s inequality to the term $\|x_{k+1}-x_k\| \Big( L_{F,x}\|x_k - x_{k-1}\| + L_{F,\theta}\|\theta_{k}-\theta_{k-1}\|\Big)$, we have
\begin{align*}
    &\|x_{k+1}-x_k\| \Big( L_{F,x}\|x_k - x_{k-1}\| + L_{F,\theta}\|\theta_{k}-\theta_{k-1}\|\Big)\\
    & \quad\leq \frac{L_{F,x}}{2} \|x_{k+1}-x_k\|^2 + \frac{L_{F,x}}{2} \|x_k - x_{k-1}\|^2  +\frac{L_{F,\theta}}{2} \|x_{k+1}-x_k\|^2 + \frac{L_{F,\theta}}{2} \|\theta_k - \theta_{k-1}\|^2 \notag \\
    & \quad =\left(\frac{L_{F,x}+ L_{F,\theta}}{2} \right) \|x_{k+1}-x_k\|^2 + \frac{L_{F,x}}{2} \|x_k - x_{k-1}\|^2 +  \frac{L_{F,\theta}}{2} \|\theta_k - \theta_{k-1}\|^2.
\end{align*}
Now let us add and subtract $\frac{L_{F,x}}{2} \|x_{k+1} - x_k\|^2$ and $\frac{L_{F,\theta}}{2} \|\theta_{k+1} - \theta_k\|^2$ with the respective terms to obtain
\begin{align*}
    &\left(\frac{L_{F,x}+ L_{F,\theta}}{2} \right) \|x_{k+1}-x_k\|^2 + \frac{L_{F,x}}{2} \|x_k - x_{k-1}\|^2 +  \frac{L_{F,\theta}}{2} \|\theta_k - \theta_{k-1}\|^2 \notag\\
    &\quad = \left(\frac{2 L_{F,x}+ L_{F,\theta}}{2} \right) \|x_{k+1} - x_k\|^2  +\frac{L_{F,\theta}}{2} \|\theta_{k+1} - \theta_k\|^2 \notag \\
    &
\qquad +  \frac{L_{F,x}}{2} \left( \|x_k - x_{k-1}\|^2 - \|x_{k+1} - x_k\|^2 \right) + \frac{L_{F,\theta}}{2} \left( \|\theta_k - \theta_{k-1}\|^2 - \|\theta_{k+1} - \theta_k\|^2 \right).
\end{align*}
Putting all terms together in \eqref{eq:prop-p6}, we get
\begin{align}\label{eq:prop-p7}
    &\left(x_{k+1}-x\right)^{\top} F\left(x,\theta^*\right) + \sum_{j=1}^J\lambda^{(j)}f_j(x_{k+1},\theta^*) - \Phi_{\rho_k}(x,\lambda_k,\theta^*) & \notag\\
    & \quad \leq \frac{1}{2\gamma_k}\Big(\|x_k - x\|^2 - \|x_{k+1} - x\|^2\Big) + \frac{1}{2\rho_k}\Big(\|\lambda_k-\lambda\|^2 - \| \lambda_{k+1}-\lambda\|^2 \Big)  +\frac{J \rho_k}{2} L^2_{f,\theta} \|\theta_k - \theta^*\|^2\notag \\
    &\qquad +\left(\frac{\rho_k C_1 +\|\lambda^*\|C_2+2 L_{F,x}+ L_{F,\theta}}{2}  - \frac{1}{2\gamma_k} \right) \|x_{k+1} - x_k\|^2 +\frac{C_2}{2} \|\lambda_k-\lambda^*\| \|x_{k+1}-x_k\|^2 +\frac{L_{F,\theta}}{2} \|\theta_{k+1} - \theta_k\|^2 \notag\\
&\qquad + \langle r_{k+1}, x_{k+1}-x\rangle - \langle r_k, x_k-x\rangle  + L_{F,\theta}\| \theta_{k+1}-\theta^*\| \|x_{k+1}-x\| + L^{\Phi}_{x \theta}\|\theta_k - \theta^*\| \Big(\|x_{k+1}-x_k\| + \|x_k - x\| \Big) \notag \\
&\qquad + \frac{L_{F,x}}{2} \left( \|x_k - x_{k-1}\|^2 - \|x_{k+1} - x_k\|^2 \right) + \frac{L_{F,\theta}}{2} \left( \|\theta_k - \theta_{k-1}\|^2 - \|\theta_{k+1} - \theta_k\|^2 \right) \notag \\
&\qquad + L_{\lambda \theta}^{\Phi}\|\lambda_{k+1}-\lambda\|\|\theta_k-\theta^*\| +  2L_{f,\theta}(J\rho_k D_f+\sqrt{J}\|\lambda_k\|)\|\theta_k-\theta^*\|.
\end{align}
\end{proof}

\section{Proof of Lemma \ref{prop:prop2}}\label{SecC}
\begin{proof}{Proof} From \eqref{eq:prop-p7}, setting $(x,\lambda)=(x^*, \lambda^*)$ gives, for all $k\geq 0$
\begin{align}\label{eq:prop2.1}
    & 0 \leq \frac{1}{2\gamma_k}\Big(\|x_k - x^*\|^2 - \|x_{k+1} - x^*\|^2\Big) + \frac{1}{2\rho_k}\Big(\|\lambda_k-\lambda^*\|^2 - \| \lambda_{k+1}-\lambda^*\|^2 \Big)  +\frac{L_{F,\theta}}{2} \|\theta_{k+1} - \theta_k\|^2\notag \\
    &\qquad + \left(\frac{\rho_k C_1 +\|\lambda^*\|C_2+2 L_{F,x}+ L_{F,\theta}}{2}  - \frac{1}{2\gamma_k} \right) \|x_{k+1} - x_k\|^2 +\frac{C_2}{2} \|\lambda_k-\lambda^*\| \|x_{k+1}-x_k\|^2 - \langle r_k, x_k-x^*\rangle  \notag\\
&\qquad + \langle r_{k+1}, x_{k+1}-x^*\rangle  + L_{F,\theta}\| \theta_{k+1}-\theta^*\| \|x_{k+1}-x^*\| + L^{\Phi}_{x \theta}\|\theta_k - \theta^*\| \Big(\|x_{k+1}-x_k\| + \|x_k - x^*\| \Big) \notag \\
&\qquad + \frac{L_{F,x}}{2} \left( \|x_k - x_{k-1}\|^2 - \|x_{k+1} - x_k\|^2 \right) + \frac{L_{F,\theta}}{2} \left( \|\theta_k - \theta_{k-1}\|^2 - \|\theta_{k+1} - \theta_k\|^2 \right) \notag \\
&\qquad + L_{\lambda \theta}^{\Phi}\|\lambda_{k+1}-\lambda^*\|\|\theta_k-\theta^*\| +\frac{J \rho_k}{2} L^2_{f,\theta} \|\theta_k - \theta^*\|^2 +  2L_{f,\theta}(J\rho_k D_f+\sqrt{J}\|\lambda_k\|)\|\theta_k-\theta^*\|.
\end{align}
To bound the cross term $L_{\lambda \theta}^{\Phi}\|\lambda_{k+1}-\lambda^*\|\|\theta_k-\theta^*\|$, we apply Young’s inequality for any $\alpha_k>0$
\begin{align*}     L_{\lambda\theta}^{\Phi}\|\lambda_{k+1}-\lambda^*\| \|\theta_k-\theta^*\| \leq \frac{\alpha_k L^\Phi_{\lambda \theta}}{2} \|\lambda_{k+1}-\lambda^*\|^2 + \frac{L^\Phi_{\lambda \theta}}{2\alpha_k }  \|\theta_k-\theta^*\|^2.
\end{align*}
Using the above inequality in \eqref{eq:prop2.1}, multiplying both sides by $2\rho_k$ and rearranging the terms, we obtain
\begin{align}\label{lambda bound}
    \left(1-\rho_k L^\Phi_{\lambda\theta} \alpha_k  \right)  \|\lambda_{k+1}-\lambda^*\|^2 & \leq  \|\lambda_k-\lambda^*\|^2 + \rho_k\left(\rho_k C_1 +\|\lambda^*\|C_2+2 L_{F,x}+ L_{F,\theta}  - \frac{1}{\gamma_k} \right) \|x_{k+1} - x_k\|^2 \notag\\
    & \qquad  + \rho_k C_2 \|\lambda_k-\lambda^*\| \|x_{k+1}-x_k\|^2 + 2 \rho_k \left(T_k + R_k\right),
\end{align}
where 
\begin{align*}
T_k &\triangleq \tfrac{1}{2\gamma_k}\left(\|x_k-x^*\|^2-\|x_{k+1}-x^*\|^2\right)
+ \tfrac{L_{F,x}}{2} \left( \|x_k - x_{k-1}\|^2 - \|x_{k+1} - x_k\|^2 \right)
+ \langle r_{k+1},x_{k+1}-x^*\rangle  -\langle r_k,x_k-x^*\rangle,\\
R_k &\triangleq L_{F,\theta}\|\theta_{k+1}-\theta^*\| \|x_{k+1}-x^*\|
+ L^{\Phi}_{x\theta}\|\theta_k - \theta^*\| \big(\|x_{k+1}-x_k\| + \|x_k - x^*\| \big) + \tfrac{L_{F,\theta}}{2} \|\theta_k - \theta_{k-1}\|^2 \\
& \quad + \tfrac{L^\Phi_{\lambda\theta}}{2\alpha_k}\|\theta_k-\theta^*\|^2
+\frac{J \rho_k}{2} L^2_{f,\theta} \|\theta_k - \theta^*\|^2 +  2L_{f,\theta}(J\rho_k D_f+\sqrt{J}\|\lambda_k\|)\|\theta_k-\theta^*\|.
\end{align*}
Assume now that $0<\rho_k L_{\lambda\theta}^{\Phi} \alpha_k < 1$. Then, dividing both sides of \eqref{lambda bound} by $ \left(1-  \rho_k L^\Phi_{\lambda\theta}  \alpha_k  \right)$ we get
\begin{align}\label{eq:prop2.2.1}
 \|\lambda_{k+1}-\lambda^*\|^2 & \leq \frac{1}{    \left(1-\rho_k L^\Phi_{\lambda\theta} \alpha_k  \right)}  \|\lambda_k-\lambda^*\|^2 + \frac{\rho_k}{\left(1-\rho_k L^\Phi_{\lambda\theta} \alpha_k  \right)} C_2 \|\lambda_k-\lambda^*\| \|x_{k+1}-x_k\|^2  \notag\\
    & \qquad  + \frac{\rho_k}{\left(1-\rho_k L^\Phi_{\lambda\theta} \alpha_k  \right)}\left(\rho_k C_1 +\|\lambda^*\|C_2+2 L_{F,x}+ L_{F,\theta}  - \frac{1}{\gamma_k} \right) \|x_{k+1} - x_k\|^2 \notag\\
    & \qquad + \frac{2\rho_k}{\left(1-\rho_k L^\Phi_{\lambda\theta} \alpha_k  \right)} \left(T_k + R_k\right).
\end{align}
We now provide upper bounds the auxiliary terms $T_k$ and $R_k$. For $T_k$ summing from $k=0, \dots, K-1$, using $x_{-1}=x_0$ and $r_0=0$, and $\gamma_k=\gamma$ we get
\begin{align}\label{eq:prop2.3}
\sum_{k=0}^{K-1}T_k
&=\frac1{2\gamma}\Bigl(\|x_0-x^*\|^2-\|x_K-x^*\|^2\Bigr) +\frac{L_{F,x}}2\bigl(\|x_0 - x_{-1}\|^2 - \|x_K - x_{K-1}\|^2\bigr) \notag\\
&\quad +\sum_{k=0}^{K-1}\Bigl(\Bigl\langle r_{k+1},x_{k+1}-x^*\Bigr\rangle-\Bigl\langle r_k,x_k-x^*\Bigr\rangle \Bigr) \notag\\
&= \frac1{2\gamma}\Bigl(\|x_0-x^*\|^2-\|x_K-x^*\|^2\Bigr) +\frac{L_{F,x}}2\bigl(\|x_0 - x_{-1}\|^2 - \|x_K - x_{K-1}\|^2\bigr)+ \Bigl\langle r_K,x_K-x^*\Bigr\rangle.
\end{align}
By Lipschitzness of $F$ in $(x,\theta)$ and recalling that
$
  r_K=F(x_K,\theta_K)-F(x_{K-1},\theta_{K-1})$, one can easily show that $ \|r_K\|\le L_{F,x}\|x_K-x_{K-1}\|+L_{F,\theta}\|\theta_K-\theta_{K-1}\|.$
Hence, by Cauchy–Schwarz and Young’s inequality
\begin{align*}
     \langle r_K,x_K-x^*\rangle & \leq\left( L_{F,x}\|x_K-x_{K-1}\|+L_{F,\theta}\|\theta_K- \theta_{K-1}\|\right) \|x_K-x^*\|\\
   &\leq \tfrac{L_{F,x}}{2}\|x_K-x_{K-1}\|^2
      +\tfrac{L_{F,x}}{2}\|x_K-x^*\|^2
      +\tfrac{L_{F,\theta}}{2}\|\theta_K-\theta_{K-1}\|^2
      +\tfrac{L_{F,\theta}}2\|x_K-x^*\|^2 \\
     &=  \tfrac{L_{F,x}}{2}\|x_K-x_{K-1}\|^2
      +\tfrac{L_{F,\theta}}{2}\|\theta_K-\theta_{K-1}\|^2
      +\tfrac{L_{F,x}+L_{F,\theta}}{2}\|x_K-x^*\|^2 .
\end{align*}
From the definition of $T_k$ and telescoping with $x_{-1}=x_0$ and $r_0=0$ in \eqref{eq:prop2.3}, we have
\begin{align} \label{eq:prop2.3.1}
    \sum_{k=0}^{K-1}T_k &\leq 
       \tfrac{1}{2\gamma}\|x_0-x^*\|^2 +\left( \tfrac{L_{F,x}+L_{F,\theta}}{2}-\tfrac{1}{2\gamma} \right) \|x_K-x^*\|^2 +\tfrac{L_{F,\theta}}{2}\|\theta_K-\theta_{K-1}\|^2.
\end{align}
Choosing $\gamma \in\Big(0,(L_{F,x}+L_{F,\theta})^{-1}\Big)$ the coefficient of $\|x_K-x^*\|^2$ on the right-hand side of \eqref{eq:prop2.3.1} is non-positive and may be discarded. Hence,  $\displaystyle \sum_{k=0}^{K-1}T_k \le \frac{1}{2\gamma} \|x_0-x^*\|^2 + \frac{L_{F,\theta}}{2} \|\theta_K-\theta_{K-1}\|^2.$
By the compactness of $\Theta$, $\|\theta_K-\theta_{K-1}\|\le\|\theta_K-\theta^*\|+\|\theta_{K-1}-\theta^*\|\le 2D_{\Theta}$, it follows that
\begin{align*}
     \sum_{k=0}^{K-1}T_k  \leq \frac{1}{2\gamma} \|x_0-x^*\|^2 + 2L_{F,\theta}D_{\Theta}^2 =: C_T,
\end{align*}
where $C_T<\infty$ is independent of $K$. Therefore, for every $K\ge 1$, $\sum_{k=0}^{K-1}T_k\le C_T.$

Next, we estimate $R_k$. Again, by the compactness of $X$, $\|x_{k+1}-x^*\|\leq D_{x}$ and $\|x_{k+1}-x_k\|\leq \|x_{k+1}-x^*\|+\|x_k-x^*\|\leq 2D_{x}$. Therefore, for every $k\geq 0$,
\begin{align}\label{eq:prop2.4}
    R_k
&\leq L_{F,\theta}D_{X}\|\theta_{k+1}-\theta^*\|
+L_{x\theta}^{\Phi}\|\theta_k-\theta^*\|\big(2D_{X}+D_{X}\big)
+\frac{L_{F,\theta}}{2}\|\theta_k-\theta_{k-1}\|^2 \notag\\
&\quad +\frac{L_{\lambda\theta}^{\Phi}}{2\alpha_k}\|\theta_k-\theta^*\|^2
+\frac{J\rho_k}{2}L_{f,\theta}^2\|\theta_k-\theta^*\|^2
+2L_{f,\theta}\big(J\rho_k D_f+\sqrt{J}\|\lambda_k\|\big)\|\theta_k-\theta^*\| \notag\\
&\le L_{F,\theta}D_{X}\|\theta_{k+1}-\theta^*\|
+\Big(3L_{x\theta}^{\Phi}D_{X}+2J\rho_k D_fL_{f,\theta}\Big)\|\theta_k-\theta^*\|
+\frac{L_{F,\theta}}{2}\|\theta_k-\theta_{k-1}\|^2 \notag\\
&\quad +\left(\frac{L_{\lambda\theta}^{\Phi}}{2\alpha_k}+\frac{J\rho_k}{2}L_{f,\theta}^2\right)\|\theta_k-\theta^*\|^2
+2\sqrt{J}L_{f,\theta}\|\lambda_k\|\|\theta_k-\theta^*\|.
\end{align}
Using $\|\theta_k-\theta_{k-1}\|^2\le 2\|\theta_k-\theta^*\|^2+2\|\theta_{k-1}-\theta^*\|^2,$ we further obtain
\begin{align}\label{eq:prop2.4b}
R_k
&\le L_{F,\theta}D_{X}\|\theta_{k+1}-\theta^*\|
+\Big(3L_{x\theta}^{\Phi}D_{X}+2J\rho_k D_fL_{f,\theta}\Big)\|\theta_k-\theta^*\| \notag\\
&\quad +L_{F,\theta}\big(\|\theta_k-\theta^*\|^2+\|\theta_{k-1}-\theta^*\|^2\big)
+\left(\frac{L_{\lambda\theta}^{\Phi}}{2\alpha_k}+\frac{J\rho_k}{2}L_{f,\theta}^2\right)\|\theta_k-\theta^*\|^2 \notag\\
&\quad +2\sqrt{J}L_{f,\theta}\|\lambda_k\|\,\|\theta_k-\theta^*\|.
\end{align}
From the induction hypothesis, we have that $\|\lambda_k-\lambda^*\|\leq D_\Lambda$, $k=0,1,\dots,K-1$, hence, by the triangle inequality, $\|\lambda_k\|\leq \|\lambda_k-\lambda^*\|+\|\lambda^*\|
\leq D_\Lambda+\|\lambda^*\|$, $k=0,1,\dots,K-1.$ Therefore, for $k=0,1,\dots,K-1$, \eqref{eq:prop2.4b} implies
\begin{align}\label{eq:prop2.4c}
R_k
&\leq L_{F,\theta}D_{X}\|\theta_{k+1}-\theta^*\|
+\Big(3L_{x\theta}^{\Phi}D_{X}+2J\rho D_fL_{f,\theta}
+2\sqrt{J}L_{f,\theta}(D_\Lambda+\|\lambda^*\|)\Big)\|\theta_k-\theta^*\| \notag\\
&\quad +L_{F,\theta}\big(\|\theta_k-\theta^*\|^2+\|\theta_{k-1}-\theta^*\|^2\big)
+\left(\frac{L_{\lambda\theta}^{\Phi}}{2\alpha_k}+\frac{J\rho}{2}L_{f,\theta}^2\right)\|\theta_k-\theta^*\|^2.
\end{align}
Summing from $k=0$ to $(K-1)$ and noting that $\theta_{-1}=\theta_0$,
 we obtain
 \begin{align*}
      \sum_{k=0}^{K-1}R_k &\leq L_{F,\theta}D_{X}\sum_{k=1}^{K}\|\theta_k-\theta^*\| +\Big(3L_{x\theta}^{\Phi}D_{X}+2J\rho D_fL_{f,\theta}
+2\sqrt{J}L_{f,\theta}(D_\Lambda+\|\lambda^*\|)\Big)\sum_{k=0}^{K-1}\|\theta_k-\theta^*\| \\
&\quad +L_{F,\theta}\Big(\|\theta_0-\theta^*\|^2+2\sum_{k=0}^{K-1}\|\theta_k-\theta^*\|^2\Big)
+\sum_{k=0}^{K-1}\left(\frac{L_{\lambda\theta}^{\Phi}}{2\alpha_k}+\frac{J\rho}{2}L_{f,\theta}^2\right)\|\theta_k-\theta^*\|^2 =: C_R(D_\Lambda).
\end{align*}
Since $\alpha_k= \frac{1}{(k+2)^2}\in (0,1)$, 
by the learning–error summability assumption, the right-hand side is finite for each $K$, and thus $\sum_{k=0}^{K-1}R_k \leq C_R(D_\Lambda).$

Now define $\delta_k= \frac{1}{1-\rho_kL_{\lambda \theta}^{\Phi}\alpha_k}$, then \eqref{eq:prop2.2.1}, becomes
\begin{align}\label{eq:prop5a}
\| \lambda_{k+1}-\lambda^*\|^2
&\leq
\delta_k \| \lambda_k-\lambda^*\|^2
+\delta_k \rho_k C_2 \|\lambda_k-\lambda^*\| \|x_{k+1}-x_k\|^2 \notag\\
& \qquad
+\delta_k\rho_k
\left(\rho_k C_1+\| \lambda^*\|C_2+2L_{F,x}+L_{F,\theta}-\frac1{\gamma}\right)\|x_{k+1}-x_k\|^2
+2\delta_k\rho_k(T_k+R_k).
\end{align}
Next, define the sequence $\{t_k\}_{k\geq 0}$ such that $t_0=1$ and $ t_{k+1}=\frac{t_k}{\delta_k} \text{ for }k\ge 0$. Based on the parameter selections $\rho_k=\rho=\frac{1}{L_{\lambda \theta}^\Phi}$ and $\alpha_k=\frac{1}{(k+2)^2}$, we conclude that $t_k=\frac{t_0(k+2)}{2(k+1)}$ which implies that $t_k\in [\frac{1}{2},1]$ for any $k\geq 0$. Multiplying both sides of \eqref{eq:prop5a} by $t_{k+1}$, we obtain
\begin{align}\label{eq:eq:prop5b}
    t_{k+1}\|\lambda_{k+1}-\lambda^*\|^2
&\leq
t_k\|\lambda_k-\lambda^*\|^2
+\rho_k  t_k C_2 \|\lambda_k-\lambda^*\| \|x_{k+1}-x_k\|^2 \notag\\
& \qquad  + \rho_k t_k
\left(\rho_k C_1+\|\lambda^*\|C_2+2L_{F,x}+L_{F,\theta}-\frac1{\gamma}\right)\|x_{k+1}-x_k\|^2
+2\rho_k t_{k}(T_k+R_k).
\end{align}
We prove the boundedness of the dual iterates by induction, showing that there exists $D_\Lambda>0$ such that $\|\lambda_k-\lambda^*\|\leq D_\Lambda$ for all $k\ge 0.$ Assume, as the induction hypothesis, that $\|\lambda_k-\lambda^*\|\leq D_\Lambda \quad\text{for all }k=0,1,\dots,K-1.$ Hence, we get from \eqref{eq:eq:prop5b}
\begin{align}\label{eq:eq:prop5c}
    t_{k+1}\|\lambda_{k+1}-\lambda^*\|^2
&\leq
t_k\|\lambda_k-\lambda^*\|^2
+\rho_k t_k\left(\rho_k C_1+\|\lambda^*\|C_2+2L_{F,x}+L_{F,\theta}+C_2D_\Lambda-\frac{1}{\gamma}\right)\|x_{k+1}-x_k\|^2
\notag\\
&\quad
+2\rho_k t_k(T_k+R_k).
\end{align}
Now, for any $\rho_k \equiv \rho >0$, let us choose $\gamma$ such that $\gamma
\leq
(\rho C_1+\|\lambda^*\|C_2+2L_{F,x}+L_{F,\theta}+C_2D_\Lambda)^{-1}$, then the coefficient of $\|x_{k+1}-x_k\|^2$ in \eqref{eq:eq:prop5c} is non-positive and may be discarded. Therefore, we have
\begin{align*}
    t_{k+1}\|\lambda_{k+1}-\lambda^*\|^2
\leq
t_k\|\lambda_k-\lambda^*\|^2
+2\rho t_k(T_k+R_k).
\end{align*}
Summing this inequality from $k=0$ to $K-1$ yields
\begin{align*}
    t_K\|\lambda_K-\lambda^*\|^2
\leq
t_0\|\lambda_0-\lambda^*\|^2
+2\rho\sum_{k=0}^{K-1}t_k(T_k+R_k).
\end{align*}
Since $t_k\in [\frac{1}{2},1]$ for all $k\geq 0$, it follows that
\begin{align*}
    \frac{1}{2}\|\lambda_K-\lambda^*\|^2
&\leq
\|\lambda_0-\lambda^*\|^2
+2\rho\sum_{k=0}^{K-1}(T_k+R_k) \notag\\
&\leq
\|\lambda_0-\lambda^*\|^2+2\rho\bigl(C_T+C_R(D_\Lambda)\bigr),
\end{align*}
where the last inequality follows from the fact that $\sum_{k=0}^{K-1}T_k\le C_T$ and $\sum_{k=0}^{K-1}R_k \leq C_R(D_\Lambda).$ Thus, we have
\begin{align*}
    \|\lambda_K-\lambda^*\|^2
\leq
2\|\lambda_0-\lambda^*\|^2+4\rho\left(C_T+C_R(D_\Lambda)\right)
=2\|\lambda_0-\lambda^*\|^2+\frac{2\rho}{\gamma} \|x_0-x^*\|^2 + 8\rho L_{F,\theta}D_{\Theta}^2+4\rho C_R\left(D_\Lambda \right),
\end{align*}
where in the last equality we used the definition of $C_T$. Since both $\gamma = \mathcal O(D_\Lambda)$ and $C_R\left(D_\Lambda \right)=\mathcal O(D_\Lambda)$, selecting $D_\Lambda\triangleq B + \sqrt{A}$, where $A \triangleq 
2\|\lambda_0-\lambda^*\|^2
+
2\rho \|x_0-x^*\|^2\bigl(\rho C_1+\|\lambda^*\|C_2+2L_{F,x}+L_{F,\theta}\bigr)
+
8\rho L_{F,\theta}D_\Theta^2
+4\rho\Bigg(
\Bigl(L_{F,\theta}D_X+3L_{x\theta}^{\Phi}D_X+2J\rho D_fL_{f,\theta}+2\sqrt{J}L_{f,\theta}\|\lambda^*\|\Bigr)
\sum_{k=0}^{\infty}\|\theta_k-\theta^*\| +L_{F,\theta}\Bigl(D_\Theta^2+2\sum_{k=0}^{\infty}\|\theta_k-\theta^*\|^2\Bigr) 
+\frac{L_{\lambda\theta}^{\Phi}}{2}\sum_{k=0}^{\infty}\frac{\|\theta_k-\theta^*\|^2}{\alpha_k}
+\frac{J\rho}{2}L_{f,\theta}^2\sum_{k=0}^{\infty}\|\theta_k-\theta^*\|^2
\Bigg)$ and $B \triangleq 2\rho C_2\|x_0-x^*\|^2+8\rho\sqrt{J}\,L_{f,\theta}\sum_{k=0}^{\infty}\|\theta_k-\theta^*\|$, implies that $\|\lambda_K-\lambda^*\|^2\leq D_\Lambda^2$ which completes the induction, hence, the result is proved. 
\end{proof}

\section{Proof of Theorem \ref{Thm:thm1}}\label{SecD}
\begin{proof}{Proof} 
From \eqref{eq:prop-p7}, for every $k\ge0$, $x\in X$, $\lambda \in \mathbb R^J_+ $, we have
\begin{align}\label{eq:thm2}
(x_{k+1}-x)^\top F(x,\theta^*)+ \sum_{j=1}^J \lambda^{(j)} f_j(x_{k+1},\theta^*)-\Phi_\rho(x,\lambda_k,\theta^*)
\le \mathcal{T}_k(x)+ \mathcal{R}_k(x,\lambda)+\frac{1}{2\rho}\bigl(\|\lambda_k-\lambda\|^2-\|\lambda_{k+1}-\lambda\|^2\bigr),
\end{align}
where we grouped terms as follows 
\begin{align*}
    \mathcal{T}_k(x)
&\triangleq \frac{1}{2\gamma}\bigl(\|x_k-x\|^2-\|x_{k+1}-x\|^2\bigr)
+ \frac{L_{F,x}}{2}\big(\|x_k-x_{k-1}\|^2-\|x_{k+1}-x_k\|^2\big)\\
&\quad + \big\langle r_{k+1},x_{k+1}-x\big\rangle - \big\langle r_k,x_k-x\big\rangle,\\
\mathcal{R}_k(x,\lambda)&\triangleq  
\left(\frac{\rho_k C_1 +\|\lambda^*\|C_2+2 L_{F,x}+ L_{F,\theta}}{2}  - \frac{1}{2\gamma_k} \right) \|x_{k+1} - x_k\|^2 +\frac{C_2}{2} \|\lambda_k-\lambda^*\| \|x_{k+1}-x_k\|^2 \\
&\quad + \tfrac{L_{F,\theta}}{2}\|\theta_{k+1}-\theta_k\|^2 + \tfrac{L_{F,\theta}}{2}\big(\|\theta_k-\theta_{k-1}\|^2-\|\theta_{k+1}-\theta_k\|^2\big)
+ L_{F,\theta}\|\theta_{k+1}-\theta^*\|\|x_{k+1}-x\| \\
&\quad + L^\Phi_{x\theta}\|\theta_k-\theta^*\|\big(\|x_{k+1}-x_k\|+\|x_k-x\|\big)
+ L^\Phi_{\lambda\theta}\|\lambda_{k+1}-\lambda\|\|\theta_k-\theta^*\| \\
& \quad +  \frac{J\rho}{2}L_{f,\theta}^2\|\theta_k-\theta^*\|^2
+ 2L_{f,\theta}(J\rho D_f+\sqrt{J}\|\lambda_k\|)\|\theta_k-\theta^*\|.
\end{align*}
Moving $-\Phi_\rho(x,\lambda_k,\theta^*)$ to the right in \eqref{eq:thm2}, then summing over $k=0, \cdots, K-1$ and dividing by $K$, and using convexity of each $f_j(\cdot, \theta^*)$ to pass the average inside, we get
\begin{align}\label{eq:thm3}
   (\bar{x}_K-x)^\top F(x,\theta^*)+  f(\bar{x}_K, \theta^*)^\top \lambda &\leq \frac{1}{K} \sum_{k=0}^{K-1} \Phi_\rho(x,\lambda_k,\theta^*) + \frac{1}{K} \sum_{k=0}^{K-1} \bigl(\mathcal{T}_k(x) +\mathcal{R}_k(x,\lambda)\bigr) \nonumber \\
   &\quad + \frac{1}{2\rho} \bigl( \|\lambda_0-\lambda\|^2 - \|\lambda_K-\lambda\|^2  \bigr).
\end{align}

Considering $\sum_{k=0}^{K-1} \mathcal T_k(x)$, telescoping in the $x$-terms and the $r_k$-pair, we have
\begin{align*}
    \sum_{k=0}^{K-1} \mathcal T_k(x)= \frac{1}{2\gamma}\left(\|x_0-x\|^2 -\|x_K-x\|^2 \right)-\frac{L_{F,x}}{2}\|x_K-x_{K-1}\|^2+ \big\langle r_K,x_K-x\big\rangle.
\end{align*}
Note that $\|r_K\|\leq L_{F,x}\|x_K-x_{K-1}\|+L_{F,\theta}\|\theta_K-\theta_{K-1}\|$, therefore $\big\langle r_K,x_K-x\big\rangle \leq \frac{L_{F,x}}{2}\|x_K-x_{K-1}\|^2+\frac{L_{F,x}}{2}\|x_K-x\|^2+ L_{F,\theta} \|\theta_K-\theta_{K-1}\|\|x_K-x\|$. Substituting this into the previous equality, we have 
\begin{align*}
    \sum_{k=0}^{K-1} \mathcal T_k(x)= \frac{1}{2\gamma}\|x_0-x\|^2 + \left( \frac{L_{F,x}}{2}-\frac{1}{2\gamma}\right)\|x_K-x\|^2 +L_{F,\theta}\|\theta_K-\theta_{K-1}\| \|x_K-x\|.
\end{align*}
Choosing $\gamma \in \left(0, L_{F,x}^{-1}\right)$, the coefficient of $\|x_K-x\|^2$ is non-positive and can be discarded. Since $x, x_o, x_K \in X$ and $\theta_K, \theta_{K-1} \in \Theta$, we have $\|x_0-x\|\leq D_X, \quad \|x_K-x\|\leq D_X$, and $\|\theta_K-\theta_{K-1}\| \leq 2D_\Theta$. Therefore,
\begin{align}\label{eq:thm5}
    \sum_{k=0}^{K-1}\mathcal T_k(x) \leq   \tfrac{1}{2\gamma} D_X^2 + 2L_{F,\theta}D_XD_\Theta \triangleq C'_T
\quad \forall K, \quad \forall x\in X.
\end{align}

Next, we focus on $\sum_{k=0}^{K-1} \mathcal{R}_k(x,\lambda)$. By continuity on the compact sets $X,\Theta$ we have that $\sup_{x\in X,\theta\in\Theta}|f_j(x, \theta^*)|\le D_f$ for all $j=1, \cdots, J$. Now, let the constant-parameters $\rho_k\equiv \rho>0$ and $\gamma_k\equiv \gamma>0$. By Lemma \ref{prop:prop2}, the dual sequence is bounded, and hence there exists $D_\Lambda>0$ such that $\|\lambda_k-\lambda^*\|\le D_\Lambda$ for all $k\ge 0.$ Therefore, $\frac{C_2}{2}\|\lambda_k-\lambda^*\|\,\|x_{k+1}-x_k\|^2
\leq
\frac{C_2D_\Lambda}{2}\|x_{k+1}-x_k\|^2$. Consequently, if the primal step-size is chosen so that $\gamma \in \Bigl(0,\,
\bigl(\rho C_1+\|\lambda^*\|C_2+2L_{F,x}+L_{F,\theta}+C_2D_\Lambda\bigr)^{-1}
\Bigr),$ then the whole coefficient of $\|x_{k+1}-x_k\|^2$ in $\mathcal R_k(x,\lambda)$ is non-positive and may be discarded. The remaining terms in $\mathcal R_k(x,\lambda)$ are controlled by the boundedness of $X$, $\Theta$, and $\{\lambda_k\}_{k\geq 0}$ together with the learning-error summability assumption. In particular, since $\|\lambda_k\|\leq \|\lambda_k-\lambda^*\|+\|\lambda^*\|\leq D_\Lambda+\|\lambda^*\|$, the last term $2L_{f,\theta}(J\rho D_f+\sqrt{J}\|\lambda_k\|)\|\theta_k-\theta^*\|$ is bounded by $2L_{f,\theta}\bigl(J\rho D_f+\sqrt{J}(D_\Lambda+\|\lambda^*\|)\bigr)\|\theta_k-\theta^*\|$ which together with $\frac{J\rho}{2}L_{f,\theta}^2\|\theta_k-\theta^*\|^2$ are summable by the learning-error assumption. Therefore, we conclude that $\sum_{k=0}^{K-1}\mathcal {R}_k(x,\lambda) \le C'_R(x,\lambda)$ where 
\begin{align}\label{eq:thm6}
C'_R(x,\lambda) \triangleq &
\Big(L_{F,\theta}D_X+ 2L^\Phi_{x\theta}D_X+ 2 L^\Phi_{\lambda\theta}D_\Lambda
+2L_{f,\theta}\big(J\rho D_f+\sqrt{J}(D_\Lambda+\|\lambda^*\|)\big)\Big)
\sum_{k=0}^{\infty}\|\theta_k-\theta^*\| \notag\\
&\quad
+\left(L_{F,\theta}+\frac{J\rho}{2}L_{f,\theta}^2\right)
\sum_{k=0}^{\infty}\|\theta_k-\theta^*\|^2<+\infty. 
\end{align}
Using \eqref{eq:thm5} and \eqref{eq:thm6} in \eqref{eq:thm3}, we get
\begin{align}\label{eq:thm8}
 &(\bar{x}_K-x)^\top F(x,\theta^*) +  f(\bar{x}_K, \theta^*)^\top \lambda \nonumber\\
& \quad\leq \frac{1}{K}\sum_{k=0}^{K-1}\Phi_\rho(x,\lambda_k, \theta^*)+ \frac{C'_T + C'_R(x,\lambda)}{K} +  \frac{1}{2\rho K} \bigl( \|\lambda_0-\lambda\|^2 - \|\lambda_K-\lambda\|^2  \bigr).
\end{align}
Let $x=x^*$ within \eqref{eq:thm8}. Since $x^*\in\mathcal{X}(\theta^*)$,
we have $\Phi_\rho(x^*,\lambda_k,\theta^*)\le0$ for all $k$, hence
\begin{align}\label{eq:thm10}
(\bar{x}_K-x^*)^\top F(x^*,\theta^*)
+ f(\bar{x}_K, \theta^*)^\top \lambda
\leq
\frac{C'_T + C'_R(x^*,\lambda)}{K}
+  \frac{1}{2\rho K} \bigl( \|\lambda_0-\lambda\|^2 - \|\lambda_K-\lambda\|^2  \bigr).
\end{align}
By Lemma \ref{lemma:lem1}, evaluated at $(x^*,\lambda^*)$, we have
\begin{align}\label{eq:thm11}
(\bar{x}_K-x^*)^\top F(x^*,\theta^*)
+ [f(\bar{x}_K, \theta^*)]_{+}^\top \lambda^*
\geq 0.
\end{align}
Subtracting \eqref{eq:thm11} from \eqref{eq:thm10} eliminates
$(\bar{x}_K-x^*)^\top F(x^*,\theta^*)$ and yields
\begin{align}\label{eq:thm12}
\Big(f(\bar{x}_K, \theta^*)^\top \lambda
- [f(\bar x_K, \theta^*)]_+^\top \lambda^*\Big)
\leq
\frac{C'_T + C'_R(x^*,\lambda)}{K}
+  \frac{1}{2\rho K} \bigl( \|\lambda_0-\lambda\|^2 - \|\lambda_K-\lambda\|^2  \bigr),
\quad
\forall \lambda \in \mathbb{R}_+^J.
\end{align}
Let $\lambda=\tilde\lambda$ within \eqref{eq:thm12} such that $\displaystyle \tilde\lambda_j \triangleq
\begin{cases}
1+\lambda_j^*, & \text{if } f_j(\bar{x}_K,\theta^*)>0\\
0, & \text{otherwise}
\end{cases}$, hence,
$f(\bar{x}_K,\theta^*)^\top\tilde\lambda
- [f(\bar{x}_K, \theta^*)]_+^\top\lambda^*
= \mathbf{1}^\top [f(\bar x_K, \theta^*)]_+$ %. Using also the boundedness of $\{\lambda_k\}$,
%in \eqref{eq:thm12} with $\lambda=\tilde\lambda$ gives
which implies that 
\begin{align}\label{rq:infeasibility}
\mathbf 1^\top [f(\bar{x}_K, \theta^*)]_+
\leq
\frac{C'_T + C'_R(x^*, \tilde\lambda)}{K}
+ \frac{J+\|\lambda_0-\lambda^*\|^2}{\rho K}
= \frac{C_{\mathrm{feas}}}{K},
\end{align}
where $C_{\mathrm{feas}} = C'_T +\Big(L_{F,\theta}D_X + 2L^\Phi_{x\theta}D_X + 2 L^\Phi_{\lambda\theta}D_\Lambda
+2L_{f,\theta}\big(J\rho D_f+\sqrt{J}(D_\Lambda+\|\lambda^*\|)\big)\Big)
\sum_{k=0}^{\infty}\|\theta_k-\theta^*\| 
+\left(L_{F,\theta}+\frac{J\rho}{2}L_{f,\theta}^2\right)
\sum_{k=0}^{\infty}\|\theta_k-\theta^*\|^2
+\frac{J+\|\lambda_0-\lambda^*\|^2}{\rho}$. 

Moreover, if we define $\epsilon \triangleq  \frac{C_{\mathrm{feas}}}{K}$ and the enlarged feasible set $ \mathcal{X}_{\epsilon}(\theta^*)
= \{x\in X \mid \mathbf{1}^\top[f(x,\theta^*)]_+ \leq \epsilon\}$. Then, $\sum_{j=1}^J [f_j(x,\theta^*)]_+ \leq \epsilon$, and using the
definition of $\Phi_\rho$ together with $\|\lambda_k\|\le D_\Lambda$ yields
\begin{align*}
\Phi_\rho(x,\lambda_k,\theta^*)
&=
\sum_{j=1}^J
\frac{\bigl([\lambda_k^{(j)}+\rho f_j(x,\theta^*)]_+\bigr)^2
- \bigl(\lambda_k^{(j)}\bigr)^2}{2\rho} \\
&\le
\sum_{j=1}^J
\Bigl(\lambda_k^{(j)}[f_j(x,\theta^*)]_+
+ \frac{\rho}{2}[f_j(x,\theta^*)]_+^2\Bigr) \\
&\le
D_\Lambda\,\mathbf 1^\top [f(x,\theta^*)]_+
+ \frac{\rho}{2}\bigl(\mathbf 1^\top [f(x,\theta^*)]_+\bigr)^2 \\
&\le
D_\Lambda \epsilon
+ \frac{\rho}{2}\epsilon^2.
\end{align*}
Note that this bound is uniform in $k$. Therefore, we have
$ \frac{1}{K}\sum_{k=0}^{K-1}\Phi_\rho(x,\lambda_k,\theta^*)
\leq
D_\Lambda \epsilon
+ \frac{\rho}{2} \epsilon^2.$
Now, let $\lambda=\mathbf{0}$ in \eqref{eq:thm8} to conclude that
\begin{align}\label{eq:thm9}
(\bar{x}_K-x)^\top F(x,\theta^*)
&\leq
\frac{1}{K}\sum_{k=0}^{K-1}\Phi_\rho(x,\lambda_k,\theta^*)
+ \frac{C'_T + C'_R(x,0)}{K} \notag\\
&\leq
D_\Lambda \epsilon
+ \frac{\rho}{2} \epsilon^2 + \frac{C'_T + C'_R(x,0)}{K}\notag\\
& \leq D_\Lambda \frac{C_{\mathrm{feas}}}{K}
+ \frac{\rho}{2} \left(\frac{C_{\mathrm{feas}}}{K}\right)^2+  \frac{C'_T + C'_R(x,0)}{K}.
\end{align}
Taking the supremum over $x\in\mathcal{X}_{\epsilon_K}(\theta^*)$ gives
\begin{align}\label{eq:relaxed_gap}
    \widetilde{\mathrm{Gap}}(\bar x_K,\theta^*)
\triangleq
\sup_{x\in\mathcal{X}_{\epsilon_K}(\theta^*)}
F(x,\theta^*)^\top(\bar x_K-x)
\leq
\frac{C'_T + C'_R(x,0)}{K}
+ D_\Lambda \frac{C_{\mathrm{feas}}}{K}
+ \frac{\rho}{2}\left(\frac{C_{\mathrm{feas}}}{K}\right)^2,
\end{align}
where $C'_T$, and $C_{\mathrm{feas}}$ are defined as above and $C'_R(x,0)= \Big(L_{F,\theta}D_X + 2L^\Phi_{x\theta}D_X + 2L^\Phi_{\lambda\theta}D_\Lambda
+2L_{f,\theta}\big(J\rho D_f+\sqrt{J}(D_\Lambda+\|\lambda^*\|)\big)\Big)\sum_{k=0}^{\infty}\|\theta_k-\theta^*\|
+\left(L_{F,\theta}+\frac{J\rho}{2}L_{f,\theta}^2\right)\sum_{k=0}^{\infty}\|\theta_k-\theta^*\|^2.$
\end{proof}
\end{appendices}
\end{document}